\documentclass{amsart}
\pdfoutput=1
\usepackage{amsfonts,amscd,amssymb,amsmath,amsthm}
\usepackage{graphicx,mathrsfs,appendix,mathtools}
\usepackage{xparse,enumerate}
\usepackage[T1]{fontenc}
\usepackage[colorlinks,citecolor=blue]{hyperref}

\usepackage[utf8]{inputenc}
\usepackage[english]{babel}
\usepackage[
    backend=biber,
    style=alphabetic,
    sortlocale=auto,
    natbib=true,
    url=false, 
    doi=true,
    eprint=true,
    safeinputenc=true,
    giveninits=true,
    isbn=false,
    maxbibnames=10
]{biblatex}
\begin{document}

\newtheorem{theorem}{Theorem}[section]
\newtheorem{lemma}[theorem]{Lemma}
\newtheorem{corollary}[theorem]{Corollary}
\newtheorem{conjecture}[theorem]{Conjecture}
\newtheorem{cor}[theorem]{Corollary}
\newtheorem{proposition}[theorem]{Proposition}
\newtheorem{assumption}[theorem]{Assumption}
\theoremstyle{definition}
\newtheorem{definition}[theorem]{Definition}
\newtheorem{example}[theorem]{Example}
\newtheorem{claim}[theorem]{Claim}
\newtheorem{remark}[theorem]{Remark}

\newenvironment{pfofthm}[1]
{\par\vskip2\parsep\noindent{\sc Proof of\ #1. }}{{\hfill
$\Box$}
\par\vskip2\parsep}
\newenvironment{pfoflem}[1]
{\par\vskip2\parsep\noindent{\sc Proof of Lemma\ #1. }}{{\hfill
$\Box$}
\par\vskip2\parsep}

\newcommand{\R}{\mathbb{R}}
\newcommand{\T}{\mathcal{T}}
\newcommand{\C}{\mathcal{C}}
\newcommand{\Z}{\mathbb{Z}}
\newcommand{\Q}{\mathbb{Q}}
\newcommand{\E}{\mathbb E}
\newcommand{\N}{\mathbb N}

\newcommand{\barray}{\begin{eqnarray*}}
\newcommand{\earray}{\end{eqnarray*}}

\newcommand{\beq}{\begin{equation}}
\newcommand{\eeq}{\end{equation}}

\def\ep{\varepsilon}

\renewcommand{\Pr}{\mathbb{P}}
\newcommand{\Prob}{\Pr}
\newcommand{\Var}{\operatorname{Var}}
\newcommand{\Cov}{\operatorname{Cov}}
\newcommand{\Exp}{\mathbb{E}}
\newcommand{\expect}{\mathbb{E}}
\newcommand{\1}{\mathbf{1}}
\newcommand{\prob}{\Pr}
\newcommand{\pr}{\Pr}
\newcommand{\filt}{\mathscr{F}}
\DeclareDocumentCommand \one { o }
{%
\IfNoValueTF {#1}
{\1  }
{\1_{#1} }%
}
\newcommand{\Bernoulli}{\operatorname{Bernoulli}}
\newcommand{\Binomial}{\operatorname{Binom}}
\newcommand{\Binom}{\Binomial}
\newcommand{\Poisson}{\operatorname{Poisson}}
\newcommand{\Exponential}{\operatorname{Exp}}
\newcommand{\Unif}{\operatorname{Unif}}

\newcommand{\cadlag}{\text{c\`adl\`ag}}
\newcommand{\weakto}{\ensuremath{\Rightarrow}}

\DeclareDocumentCommand \D { o }
{%
\IfNoValueTF {#1}
{D}
{\partial_{{#1}}}
}%
\newcommand{\sgn}{\operatorname{sgn}}

\newcommand{\dsp}[2]
{%
\tfrac{n + #1}{n+#2}
}
\newcommand{\PDEOP}{\mathscr{S}^{\Psi}}
\newcommand{\PDEOPtilde}{{\widetilde{\mathscr{S}}}^{\Psi}}
\newcommand{\DSpace}{\mathcal{D}}
\newcommand{\Dspace}{\DSpace}
\newcommand{\Lloc}{L^1_{\text{loc}}}

\newcommand{\DTSpace}{\mathcal{X}}
\newcommand{\RTSpace}{\mathcal{R}}
\newcommand{\FPSpace}{\mathfrak{F}}
\newcommand{\Dto}{ \overset{\DTSpace}{\to} }
\newcommand{\aec}{\text{a.e.c.}}

\DeclareDocumentCommand \DTfn { m o O{}}
{
\IfNoValueTF {#2}
{#3 {\mathbf{ #1}}}
{#3 {#1}_{#2}}
}

\DeclareDocumentCommand \A { o }{\DTfn{A}[#1]}
\DeclareDocumentCommand \Atilde { o }{\DTfn{A}[#1][\widetilde]}
\DeclareDocumentCommand \F { o }{\DTfn{F}[#1]}
\DeclareDocumentCommand \Fhat { o }{\DTfn{F}[#1][\widehat]}
\DeclareDocumentCommand \Ftilde { o }{\DTfn{F}[#1][\widetilde]}
\DeclareDocumentCommand \G { o }{\DTfn{G}[#1]}
\DeclareDocumentCommand \Ghat { o }{\DTfn{G}[#1][\widehat]}
\DeclareDocumentCommand \Gtilde { o }{\DTfn{G}[#1][\widetilde]}
\DeclareDocumentCommand \M { o }{\DTfn{M}[#1]}
\DeclareDocumentCommand \Mtilde { o }{\DTfn{M}[#1][\widetilde]}

\DeclareDocumentCommand \diF { O{n} }{D_{#1}}
\DeclareDocumentCommand \diFtilde { O{n} }{\widetilde D_{#1}}

\DeclareDocumentCommand \DTfnsup { m m o O{}}
{
\IfNoValueTF {#3}
{ {#4 {\mathbf{ #1}}}^{{#2}} }
{ {#4 {#1}}^{{#2}}_{#3}}
}

\DeclareDocumentCommand \Ank { o }{\DTfnsup{A}{(n_k)}[#1]}
\DeclareDocumentCommand \Ankt { o }{\DTfnsup{A}{(n'_k)}[#1]}
\DeclareDocumentCommand \Mnk { o }{\DTfnsup{M}{(n_k)}[#1]}
\DeclareDocumentCommand \Mnkt { o }{\DTfnsup{M}{(n'_k)}[#1]}
\DeclareDocumentCommand \Aa { o }{\DTfnsup{A}{\alpha}[#1]}
\DeclareDocumentCommand \Aan { o }{\DTfnsup{A}{\alpha,(n)}[#1]}
\DeclareDocumentCommand \Aank { o }{\DTfnsup{A}{\alpha,(n_k)}[#1]}
\DeclareDocumentCommand \Aankt { o }{\DTfnsup{A}{\alpha,(n'_k)}[#1]}
\DeclareDocumentCommand \Aas { o }{\DTfnsup{A}{\alpha,(s)}[#1]}
\DeclareDocumentCommand \Aatilde { o }{\DTfnsup{A}{\alpha}[#1][\widetilde]}
\DeclareDocumentCommand \Aaa { o }{\DTfnsup{A}{{\alpha_+}}[#1]}
\DeclareDocumentCommand \Aaatilde { o }{\DTfnsup{A}{{\alpha_+}}[#1][\widetilde]}
\DeclareDocumentCommand \Ma { o }{\DTfnsup{M}{\alpha}[#1]}
\DeclareDocumentCommand \Man { o }{\DTfnsup{M}{\alpha,(n)}[#1]}
\DeclareDocumentCommand \Mank { o }{\DTfnsup{M}{\alpha,(n_k)}[#1]}
\DeclareDocumentCommand \Mankt { o }{\DTfnsup{M}{\alpha,(n'_k)}[#1]}
\DeclareDocumentCommand \Matilde { o }{\DTfnsup{M}{\alpha}[#1][\widetilde]}

\newcommand{\COP}{\mathscr{C}}
\newcommand{\Co}{\mathcal{C}_{0,b}}
\newcommand{\Coo}{{C}_{0}(0,\infty)}
\newcommand{\COPR}{\mathscr{R}}
\newcommand{\RG}{\mathscr{L}}
\newcommand{\RGT}{{\tilde{\mathscr{L}}}}

\newcommand{\domAG}{D_{\AG}}
\newcommand{\AG}{\mathscr{A}}
\newcommand{\domRG}{D_{\RG}}
\newcommand{\domRGB}{D_{\RGB}}
\newcommand{\RGB}{\mathscr{B}}
\newcommand{\RGK}{\mathscr{K}}
\DeclareDocumentCommand \Lpi { }
{%
\ensuremath{\operatorname{L}^1(\pi)}
}%
\DeclareDocumentCommand \Lpin { O{\cdot} }
{%
\| {#1}
\|_{\operatorname{L}^1(\pi)}
}%

\newcommand{\DTSSpace}{\mathcal{Y}}
\newcommand{\CFPSpace}{\mathfrak{F}^*}
\newcommand{\RFPSpace}{\mathfrak{R}}
\newcommand{\RFPSpacea}{\mathfrak{R}_{\operatorname{ac}}}
\newcommand{\DSto}{ \overset{\DTSSpace}{\to} }
\newcommand{\dtv}{d_{\operatorname{TV}}}

\DeclareDocumentCommand \surv { o }
{
\IfNoValueTF {#1}
{S}
{S^{#1}}
}
\DeclareDocumentCommand \tvnorm { O{\cdot} }
{%
\| {#1}
\|_{\operatorname{TV}}
}%

\DeclareDocumentCommand \candy { O{\cdot} }
{%
\| {#1}
\|_{x^{-2}}
}%
\newcommand{\dcandy}{d_{x^{-2}}}
\newcommand{\dblah}{d_{\Lloc}}

\title[Interval fragmentations with choice]{Interval fragmentations with choice: equidistribution and the evolution of tagged fragments}

\author{Pascal Maillard}
\address{Institut de Math\'ematiques de Toulouse, CNRS UMR 5921, Universit\'e de Toulouse, 118 route de Narbonne, F-31062 Toulouse Cedex 9, France}
\email{pascal.maillard at math.univ-toulouse.fr}
\author{Elliot Paquette}
\address{Department of Mathematics, McGill University \\
  Burnside Hall, Room 105 \\
  805 Sherbrooke Street W \\ 
  Montreal, Quebec, Canada}
\email{elliot.paquette@gmail.com}
\date{June 29, 2020}

\begin{abstract}
We consider a Markovian evolution on point processes, the $\Psi$--process, on the unit interval in which points are added according to a rule that depends only on the spacings of the existing point configuration.  Having chosen a spacing, a new point is added uniformly within it. Building on previous work of the authors and of Junge, we show that the empirical distribution of points in such a process is always equidistributed under mild assumptions on the rule, generalizing work of Junge.

A major portion of this article is devoted to the study of a particular growth--fragmentation process, or cell process, which is a type of piecewise--deterministic Markov process (PDMP).  This process represents a linearized version of a size--biased sampling from the $\Psi$--process.  We show that this PDMP is ergodic and develop the semigroup theory of it, to show that it describes a linearized version of the $\Psi$--process.  This PDMP has appeared in other contexts, and in some sense we develop its theory under minimal assumptions.

\end{abstract}

\maketitle

\section{Introduction}
\label{sec:introduction}

In this article, we revisit the $\Psi$--process introduced in \cite{candy} and studied in \cite{Junge}.  This point process on $[0,1]$ can be defined as follows.  Suppose that $N_t$ is the counting function for a Poisson process on $[0, \infty)$ with intensity $e^t.$  Suppose at time $0,$ the process is started at some discrete point configuration on $[0,1].$  A point is added to the configuration on $[0,1]$ at each time that $N_t$ jumps according to a Markovian rule we will now describe.  

At each time $t,$ the point configuration partitions $[0,1]$ into intervals, whose lengths we denote by
\(
I_1^{(t)},
I_2^{(t)},
\ldots,
I_{N_t+n_0}^{(t)}
\).
Define the size-biased empirical distribution function
\[
  \Atilde[t](x) = \sum_{i=1}^{N_t+n_0} |I_i^{(t)}| \one[|I_i^{(t)}|\le x].
\]
At a jump time $t,$ select an interval $\mathcal{I}_t$ with length $\Atilde[t_{-}]^{-1}(u)$ uniformly at random, where $u$ is sampled from a law on $(0,1]$ with distribution function $\Psi$.  Finally add a point to $\mathcal{I}_t$ uniformly at random.  

Different choices of $\Psi$ produce substantially different behavior, and some canonical choices of $\Psi$ produce processes with alternative descriptions (we will elaborate on this some later; see \cite{candy} for further discussion).  The case $\Psi(u) = u$ we call the \emph{uniform} process.  In this case, intervals are selected with probability proportional to length and subdivided uniformly, so that the evolution of interval lengths $\{I_i^{(t)}\}$ has the same law as the spacings between points which are added independently and uniformly in $[0,1]$ with rate $e^t.$

The motivating example for the $\Psi$--process is the \emph{max--$k$ process}, for $k \in \N,$ which is given by $\Psi(u) = u^k.$  This process can also be described by executing the following rule whenever $N_t$ jumps: choose $k$ independent and uniformly distributed points in $[0,1].$  Select the point the lands in the largest interval, breaking ties uniformly, and add it to the existing point configuration.  Analogously, in the \emph{min--$k$ process} for $k \in \N,$ given by $\Psi(u) = 1 - (1-u)^k,$ one instead selects the point that lands in the smallest interval.  While max--$k$ and min--$k$ processes have this alternative description for $k \in \N,$ as $\Psi$--processes they are well--defined for any real $k > 0.$

The max--$k$ process is an example of a \emph{choice} algorithm, in which a constant number of equivalent random choices are presented to an agent, which then employs a heuristic to choose between them.  There are a wide variety of problems in which this has been employed, see \cite{MRS01} for a survey, and which are collectively named the \emph{power--of--choice} paradigm.  While we will not elaborate on this body of work here, let us mention there has been subsequent developments on the power--of--choice in equidistribution of points \cite{Ohadfromthefuture}.

In \cite{candy}, we show that after appropriately rescaling, the distribution of interval lengths in the $\Psi$--process converges to a deterministic limit under some mild assumptions (we give the formal statement in Theorem \ref{thm:candy} below).  Most importantly, we will suppose the following.
  \begin{assumption} \label{ass:arsch}
There are constants $c_{\Psi}>0$ and $\kappa_\Psi\in[1,\infty)$ so that
    \[
  1-\Psi(u) \ge c_{\Psi}(1-u)^{\kappa_\Psi} \quad \text{ for all } u\in(0,1).
\]
\end{assumption}
\noindent This assumption ensures that large intervals are subdivided frequently enough.  While perhaps not optimal, it can not be removed entirely; for example in the simple case that $\Psi(u) \equiv 1$ on the interval $[1-u_0,1],$ it follows the associated $\Psi$--process will cease subdividing the longest interval once its length is less than $u_0,$ which precludes any deterministic scaling limit for the empirical distribution of interval lengths.

We define
\[
  \A[t](x) = \Atilde[t](e^{t}x)
\]
for all $t,x \geq 0.$ We have the following theorem:
\begin{theorem}[\cite{candy}]
\label{thm:candy}
Assume that $\Psi$ is continuous and satisfies Assumption \ref{ass:arsch}. Then there is an absolutely continuous probability measure $\mu^\Psi$ on $(0,\infty)$ with mean 1, independent of the initial configuration, such that $\A[t]$ converges pointwise to the function $F^{\Psi}(x) = \int_0^x y\,\mu^\Psi(dy)$, almost surely as $t\to\infty$. Furthermore, the Lebesgue-Stieltjes measures $\frac{1}{z}d\A[t](z)$ almost surely converge weakly to $\mu^\Psi$ as $t\to\infty$. The function $F^\Psi$ is characterized by the equation
\[
  \frac{dF(x)}{dx} = x\int_x^\infty \frac{1}{z}\,d\Psi(F(z)).
\]
\end{theorem}
The proof of Theorem~\ref{thm:candy} in \cite{candy} relies on the solvability of a certain non-linear, non-local evolution equation together with the method of \emph{asymptotic pseudotrajectories}, or \emph{Kushner-Clark method}, for stochastic approximation algorithms, applied in an infinite-dimensional setting (see \cite{BeLeRa2002} for an earlier appearence of this method). Note that explicit convergence rates are, in general, unknown.

The max--$k$ and min--$k$ processes satisfy the assumptions of Theorem~\ref{thm:candy} for any $k \geq 1.$  The most substantial difference seen between these processes is in the tail behavior of $F^{\Psi}(x).$  For the max processes, $1-F^{\Psi}(x) \sim C_kxe^{-k x}$ as $x \to \infty$ (\cite[Proposition 9.2]{candy}) for some constant $C_k > 0$.  For the min processes with $k > 1$, $1-F^{\Psi}(x) \sim c_k x^{-1/(k-1)}.$  Nonetheless, we conjecture in \cite{candy} that the empirical distribution of points almost surely weakly converges to the uniform distribution on $[0,1]$ as $t\to\infty$, for all $\Psi$--processes to which Theorem~\ref{thm:candy} applies, in spite of the heavy--tailed behavior. A $\Psi$--process with this property is said to be \emph{equidistributed almost surely}.

In \cite{Junge}, Junge shows this for some $\Psi$--processes.
\begin{theorem}[\cite{Junge}]
  Suppose that $\Psi$ is $C^2[0,1],$ and let $\psi = \Psi'.$ If for some $\delta \in (0,1]$ and all $z \geq 0,$ 
  \[
    |z \psi'(F^{\Psi}(z)) (F^{\Psi})'(z) - \psi(F^{\Psi}(z))| \leq (2-\delta) \psi(F^{\Psi}(z)),
  \]
  then the $\Psi$--process is equidistributed almost surely as $t \to \infty.$
  \label{thm:Junge}
\end{theorem}
This condition is satisfied by the max--$2$ process but surprisingly not by any max--$k$ process for $k \geq 3,$ nor any min--$k$ process.  It is however satisfied by certain combinations of max--$k$ and min--$k$ process, including ones which slightly favor picking a smaller interval.

In this paper, we remove the additional technical condition of \cite{Junge} and weaken the smoothness assumptions on $\Psi:$
\begin{theorem}
  Suppose that $\Psi$ is $C^1[0,1]$ and satisfies Assumption \ref{ass:arsch}.
  Then the $\Psi$--process is equidistributed almost surely as $t \to \infty.$
  \label{thm:main}
\end{theorem}
\noindent This in particular shows that all max--$k$ and min--$k$ processes equidistribute for $k \geq 1$.

The method we use is a development on top of \cite{Junge}, and we outline this method and how our method differs from \cite{Junge} (see Section \ref{sec:formalism} for a formal overview of the proof).  To show equidistribution, it suffices to show that for any $\alpha \in (0,1)$ the asymptotic fraction of points less than $\alpha$ is $\alpha.$  In turn, it should suffice to show that the empirical distribution of intervals contained in $(0,\alpha)$ is asymptotically the same as the global interval distribution.  In this case, heuristically, conditioned on picking an interval of some length, the probability that interval is in $(0,\alpha)$ is just $\alpha,$ on account of the total length of the interval being an $\alpha$ fraction of the whole; hence points would be added to $(0,\alpha)$ with asymptotic rate $\alpha e^{t}.$  Some smoothing of the interval lengths would be required to formalize such an argument, and in fact it is possible to bypass the issue entirely (see the proof of Theorem \ref{thm:main} in Section \ref{sec:formalism}).

To show that the empirical distribution of intervals in $(0,\alpha)$ has the same limit as the global distribution, we use what might be described as a linearization procedure.  We probabilistically describe the restriction of the $\Psi$--process to $(0,\alpha)$ as one in which potential points are added with rate $e^{t}$, but are thinned at a rate which is a functional of the global interval distribution of the $\Psi$--process (this is one interpretation of Proposition \ref{prop:semimartingale} or \cite[Proposition 3]{Junge} -- as a side comment, this is one place where $\Psi \in C^1[0,1]$ represents a natural technical barrier, as the description becomes problematic for rougher $\Psi$).  Knowing that the global interval distribution converges, we may seek to replace this functional by one that depends only on the \emph{limiting} interval distribution.  Having done so, we arrive at an \emph{analytic} description of an interval fragmentation process in which intervals evolve independently of one another.

The idea to make this linearized comparison is a key idea in the analysis of \cite{Junge}, and we weaken the smoothness assumptions on $\Psi$ required (see Proposition \ref{prop:continuity} and c.f.\ \cite[Proposition 5(V)]{Junge}).   Having made the reduction, the final major step is to show that the (deterministic) linearized evolution of the size--biased distribution functions converges to $F^{\Psi}.$  This is the origin of the technical hypothesis in Theorem \ref{thm:Junge}, and a new argument for this convergence is the major development in this article on Junge's and one which we believe may be of independent interest.

\subsection*{Growth--fragmentation equations}

The resulting linearized equation for the evolution of the (rescaled) size--biased distribution $\F = (\F[t])_{t \geq 0}$ is 
\begin{equation} \label{eq:RGa}
  \F[t](x)
  = {\F[0]}(e^{-t}x) + \int_0^t (e^{s-t}x)^2 \left[\int_{e^{s-t}x}^\infty \frac{R(z)}{z}\,d\F[s](z)\right]\,ds \quad \text{for all } t,x \geq 0,
\end{equation}
where $R$
is some locally integrable non--negative Borel function (for application to the $\Psi$--process, $R = \psi \circ F^{\Psi}$). 
We need to show that solutions of this equation converge for essentially arbitrary initial conditions as $t \to \infty$ to the limit:
\[
  F^R(x) = \frac{1}{Z^R} \int_0^x y\exp\left( -\int_1^x R(y)\,dy \right)\,dx,
\]
using the usual Riemann integration convention that for $x < 1,$ $\int_1^x R(y)\,dy = -\int_x^1 R(y)\,dy.$ 
The constant $Z^R$ is a normalization so that $F^R$ is a distribution function, and for $F^R$ to be a distribution function we must assume:
\begin{equation}\label{eq:ZR}
Z^R = \int_0^\infty y\exp\left( -\int_1^x R(y)\,dy \right)\,dx < \infty.
\end{equation}
Besides the finiteness of $Z^R,$ we make no further assumptions on $R.$  For the application to $\Psi$--processes, it is an automatic corollary of the theory in \cite{candy} that $Z^R$ is finite.  

The evolution equation \eqref{eq:RGa} is an integrated and scaled form of a very--well studied partial integro--differential equation, the growth--fragmentation equation, which originated in questions arising in mathematical biology (see e.g.\ \cite{Perthame,EngelNagel,CCF1,CCF2}).  
Following \cite{BertoinWatsonJFA}, the growth--fragmentation equation is given by
\begin{equation} \label{eq:gfeq}
  \partial_t u_t(x) + \partial_x( u_t(x)c(x) ) = \int_x^\infty u_t(y)k(y,x)dy - u_t(x)K(x).
\end{equation}
The quantity $u_t(x)$ represents the density of particles of varying masses $x.$  Each particle grows with a rate $c: (0,\infty) \to (0,\infty)$ that depends on its mass.  The particles fragment into smaller masses, and $k(y,x)$ represents the rate at which particles of mass $x$ result from fragmentation of a particle of mass $y > x.$ The term $K(x)$ represents the rate of destruction of particles of mass $x,$ as a result of fragmentation.
In our application, $c(x) = x$ and $k(y,x) = 2xR(y)/y,$ and $K(x) = xR(x).$  The special choice of $c(x) = x$ is is called the \emph{self--similar fragmentation equation} in \cite{CCM} (caveat emptor: \cite{CCM} use a different meaning of self--similar fragmentation kernel than \cite{BCG} or \cite{BertoinWatsonJFA}).

A common assumption in the literature is the assumption of \emph{mass--conservative fragmentations}, so that no mass is created or destroyed upon fragmentation of a particle (\cite{Bertoin,BertoinWatsonAAP,BertoinWatsonJFA,BCG,CCM}.  Analytically, this corresponds to 
\[
  K(y) = \frac{1}{y}\int_0^y x k(y,x) dx, \quad \text{ for all } y > 0.
\]
Our process \emph{does not} satisfy this assumption.  Rather, we assume \emph{population--conservative fragmentations}, that is
\[
  K(y) = \int_0^y k(y,x) dx, \quad \text{ for all } y > 0.
\]
This corresponds to the evolution of mass--biased random sampling of a fragment from a mass--conservative fragmentation.  Such fragmentation equations have been called \emph{conservative} in the literature (c.f.\ \cite{Bertoin, Bouguet}), but as there is an obvious risk of confusion, we shall enforce the lengthier \emph{population--conservative} terminology.  Besides appearing in \cite{Bouguet} explicitly, it also appears in quite general form in the work of \cite{BertoinWatsonJFA} and also in a special case (the TCP process) in \cite{Chafai,Bardet}.

The population--conservative growth--fragmentation equation (PCGFE) has been used to model the size of a biological cell, which grows in time and then undergoes mitosis (fragments) into two daughter cells, after which point \emph{one} of the resulting daughters is chosen as a representative of the two.  This naturally corresponds to a piecewise deterministic Markov process (PDMP), which grows deterministically and then jumps down. It has been called a \emph{cell process} in the literature (\cite{Bouguet,Bertoin}).  This process appears in \cite{Bouguet} (take $\tau(x)=x, \beta(x) = xR(x),$ and $Q(x,dy) = 2y\,dy$) as does it appear in \cite{BertoinWatsonJFA} ($\bar{k}(x,y) = 2yR(x)/x$ and $\underline{c}(x) = 0$).  

Our main task is to show that solutions of PCGFEs tend to equilibrium in the large--time limit.  This we do by a probabilistic method, which essentially has three components.  The first part of the work is to construct the cell process.  While general existence theorems for PDMPs exist (for example \cite{Davis1984}), they are generally formulated under the assumption of some type of uniform control on the jump rate.  As we do not assume such a feature, we must show that the cell process is well--posed and has the usual desirable properties, c.f.\ Proposition \ref{prop:markov}.

The second part of the work is to show that PCGFE solutions can be represented by the semigroup of the associated cell process  (Proposition \ref{prop:rgb}, c.f.\ Sections \ref{sec:diff_mu}, \ref{sec:construction_markov_process}, \ref{sec:Lpi_and_uniqueness}).  Under additional continuity and boundedness assumptions on $R,$ as assumed in \cite{Bouguet, BertoinWatsonJFA}, the usual $C_b(\R)$--semigroup theory, such as that which is found in \cite{EthierKurtz}, applies.  We assume very little boundedness and smoothness of $xR(x),$ and so in fact the $C_b(\R)$--semigroup formulation is nonsense (in fact, even under the additional assumption that $R$ be continuous, the $C_b(\R)$--setting is substantially complicated by the lack of boundedness on $xR(x)$).  Hence, we switch to an $\Lpi$--setting (with $\pi=dF^{R}$) to develop the semigroup theory.  This $\Lpi$--semigroup setting is common in the analysis literature on growth--fragmentation equations, see for example \cite{Mischler2016}.  Moreover, we must use some of the central ideas from \cite{Mischler2016} to develop this $\Lpi$--semigroup theory for the cell process.

The final part is to show the resulting cell process is ergodic (see Section \ref{sec:ergodicity}).  This we do by appealing to the regenerative structure of the cell process.  While we do not show a rate, as is common in much of the literature (\cite{Bouguet,BertoinWatsonJFA,BCG,CCM, Mischler2016}, we emphasize that the argument we give still works under minimal assumptions on $R.$  The argument is also conceptually simple, from the probabilistic viewpoint.  

To our knowledge, the only explicit existing work on ergodicity of PCGFEs is in \cite{Bouguet,BertoinWatsonJFA}, which would only apply to regularly varying and continuous $R$ (see \cite[Assumption 2.1]{Bouguet} and \cite[Assumption 6.1]{BertoinWatsonJFA}).  The method of \cite{Mischler2016} could in principle apply, but the details are only worked out for some explicit cases.  Let us also mention that a central idea of \cite{BertoinWatsonJFA} is something like a Girsanov transform, that allows certain mass--conservative growth--fragmentation equations (MCGFEs) to be related to other PCFGEs.  This could in priniciple allow results to be transfered between MCGFEs and PCGFEs.  We give some further discussion of our contributions to the theory of growth--fragmentation equations in Section \ref{sec:growth-fragmentation}.

\section*{Acknowledgements}

This work was initiated while the second author was invited professor (\emph{professeur invit\'e}) at Universit\'e Paris-Sud. The authors thank this institution for the support and hospitality. They would also like to thank the CRM at Universit\'e de Montr\'eal for their hospitality during a second meeting.  The authors also are indebted to Itai Benjamini, Jean Bertoin, B\'en\'edicte Haas, Matt Junge, Jean-Claude Lootgieter, Olivier Raimond and Lorenzo Zambotti for helpful conversations regarding this work.

\section{Formalism}
\label{sec:formalism}

In this section, we will give the formal setup of the $\Psi$--process and the approach to Theorem \ref{thm:main}.
Let $N_t^{\alpha}$ be the number of points that land in $[0,\alpha].$  To prove Theorem~\ref{thm:main}, it suffices to show that $e^{-t}N_t^{\alpha} \to \alpha.$
We use the convention that boldface letters represent processes, i.e.\,they are time dependent.  This parameter will appear as a subscript when referring to the process at a fixed time, e.g. $\A = (\A[t])_{t \geq 0}.$

We define the process $(\Aatilde, \Aaatilde, \Atilde)$ to be the size-biased empirical distributions of interval lengths contained in $([0,\alpha], [\alpha,1],[0,1])$, respectively.  An interval which straddles the point $\alpha$ contributes whatever fraction of its length landed on either side of $\alpha$ to $\Aatilde$ or $\Aaatilde$ respectively.  For every $t \geq 0,$ let $g_t^\alpha : \left\{ 1,2,\dots,N_t \right\} \to [0,1]$ be defined by
\[
  g_t^\alpha(j) = \frac{| I_j^{(t)} \cap [0,\alpha]|}{|I_j^{(t)}|}.
\]
We then define
\[
  \Aatilde[t](x) = \sum_{j=1}^{N_t} g_t^\alpha(j) |I_j^{(t)}| \cdot \one[|I_j^{(t)}| \leq x],
\]
for all $t,x \geq 0$ and define $\Aaatilde$ by $\Aaatilde = \Atilde - \Aatilde.$  

By comparing with the case of $\alpha = 1,$ in which $g_t^\alpha = 1,$ we see $d\Aatilde[t]$ is absolutely continuous with respect to $d\Atilde[t]$ for all $t.$  Moreover, the Radon--Nikodym derivative $\frac{d\Aatilde[t]}{d\Atilde[t]}(x)$ is equal to the fraction of mass of intervals of length $x$ that are contained in $[0,\alpha].$  In particular, we may take $\sup_{t \geq 0, x\in [0,1]} \left| \frac{d\Aatilde[t]}{d\Atilde[t]}(x)\right| \leq 1.$

It is now possible to give a semimartingale decomposition of $\Aatilde.$
\begin{proposition} 
  \label{prop:semimartingale}
  Suppose that the intervals $\left\{ I^{(0)}_j : 1 \leq j \leq n_0 \right\}$ have distinct lengths.
  Suppose that $\Psi$ is continuous. %
The process $(\Aatilde, \Atilde)$ satisfies the equation
\[
  \Aatilde[t](x)
  =
  \Aatilde[0](x)+ 
  \int_0^t e^s x^2 \int_x^\infty 
  \frac{d\Aatilde[s]}{d\Atilde[s]}(z)
  \cdot\frac{ d\Psi(\Atilde[s](z))}{z}
  \,ds + \Matilde[t](x),
\]
for all $t,x \geq 0,$ where $\Matilde$ is a martingale.
\end{proposition} 

\begin{remark}\label{rem:abscontinuity}
  If in addition $\Psi$ is absolutely continuous, then one can rewrite the expression on the right-hand side in the above proposition. First note that for any real numbers $a < b,$ and any distribution function $F,$  we have $\int \one[(a,b)]d(\Psi \circ F) \leq |\Psi(F(b)) - \Psi(F(a))|.$  It follows for any $\epsilon>0$ there is a $\delta > 0$ so that for any open set $U$ with $\int_U dF < \delta$ then $\int_U d(\Psi \circ F) \leq \epsilon.$ So, by regularity $d\Psi \circ F$ is absolutely continuous with respect to $dF.$  Hence, by the Radon--Nikodym theorem, we can also write for any $x > 0$ and any $s \geq 0,$
  \[
    \int_x^\infty 
    \frac{d\Aatilde[s]}{d\Atilde[s]}(z)
    \cdot\frac{ d\Psi(\Atilde[s](z))}{z}
    =
    \int_x^\infty 
    \frac{d\Psi(\Atilde[s])}{d\Atilde[s]}
    \cdot\frac{d\Aatilde[s]}{d\Atilde[s]}
    \cdot
    \frac{d\Atilde[s](z)}{z}
    =
    \int_x^\infty 
    \frac{d\Psi(\Atilde[s])}{d\Atilde[s]}(z)
    \cdot\frac{d\Aatilde[s](z)}{z}.
  \]

  If in addition $\Atilde[s]$ is absolutely continuous, then we may take for all $z \in [0,1]$
  \begin{equation}\label{eq:jungian}
    \frac{d\Psi(\Atilde[s])}{d\Atilde[s]}(z)
    =\psi(\Atilde[s](z)),
    \quad
    \text{where}
    \quad
    \Psi(z) = \int_0^z \psi(u)\,du.
  \end{equation}
  Note that without the assumption that $\Atilde[s]$ is absolutely continuous, we may no longer
  have \eqref{eq:jungian}.  This can be seen for example in the simple case that 
  \[
    \Atilde[s](z) = 
    \frac{1}{2}\one[z \geq 1/3]
    +\frac{1}{2}\one[z \geq 2/3]
    \quad
    \text{and}
    \quad
    \Psi(u)=u^2.
  \]
  In particular \cite[Proposition 3]{Junge} is not correct as claimed. 
  Nonetheless, the results of \cite{Junge} appear to be fixable by using Proposition \ref{prop:semimartingale}.
\end{remark}

\begin{proof}
  Under the assumption of distinct starting lengths, and from the definition of the splitting rule for the intervals (i.e.\ uniformly), almost surely all intervals $\{I^{(t)}_j : 1 \leq j \leq N_t+n_0\},$ have distinct lengths at all times $t \geq 0.$ 
  
  Define a Poisson random measure $\Pi$ on $[0,\infty) \times [0,1]^2$ with intensity $e^t \,dt \otimes d\Psi(u) \otimes dv.$  Let $\ell_t(u) = \Atilde[t^{-}]^{-1}(u)$ for all $u \in [0,1],$ with $\Atilde[t^{-}]^{-1}$ the generalized right--continuous inverse of $\Atilde[t^{-}],$ i.e.\ for all $p \in (0,1],$
    \begin{equation}\label{eq:quantile}
    \Atilde[t^{-}]^{-1}(p) = \inf \left\{ x : \Atilde[t^{-}](x) > p \right\}.
  \end{equation}
  As the lengths are distinct, we may unambiguously define a bounded non--negative measurable function $k_t^{\alpha}$ so that $k_t^{\alpha}(\ell_t(u)) = g_t^{\alpha}(j)$ where $1 \leq j \leq N_t + n_0$ is the index so that $|I_j^{(t)}| = \ell_t(u).$
  
  At a point $(s,u,v) \in \Pi$ the interval with length $\ell_s(u)$ is split into two intervals, one of fraction $v$ length of the whole.  This affects the empirical distribution function $\Atilde$ by replacing an atom at $x=\ell_s(u)$ by two atoms at $x = v\ell_s(u)$ and at $x=(1-v)\ell_s(u).$  Hence if we let $h(v,\ell,x) = v \one[\ell v \leq x] + (1-v) \one[ \ell (1-v) \leq x],$ 
  then we have the identity that for any $x \in [0,1]$ and $t \geq 0$
  \[
    \begin{aligned}
    &\Aatilde[t](x) = \Aatilde[0](x) + 
    \sum_{(s,u,v) \in \Pi, s\leq t} 
    \tilde{B}^{\alpha}_s(u,v,x), 
    \quad
    \text{where}
    \\
    &\tilde{B}^{\alpha}_s(u,v,x)
    =
    \ell_s(u) k_t^{\alpha}(\ell_s(u)) \one[ \ell_s(u) > x] h(v,\ell_s(u),x).
  \end{aligned}
  \]

  We wish to form a semimartingale decomposition, and so we will integrate $\tilde{B}^{\alpha}$ against the intensity of $\Pi.$
  We observe that for $x < \ell_s(u),$
  \[
    \int_0^1  h(v,\ell_s(u),x)\,dv = \left( \frac{x}{\ell_s(u)} \right)^2,
  \]
  and therefore
  \[
    \int_0^1
    \int_0^1
    \tilde{B}^{\alpha}_s(u,v,x)\,dv 
    d\Psi(u)
    =
    x^2
    \int_0^1
    \frac{k_t^{\alpha}(\ell_s(u))}{\ell_s(u)} \one[ \ell_s(u) > x]
    d\Psi(u).
  \]
  By the explicit definition of $\Atilde[t^{-}]^{-1}$ in \eqref{eq:quantile}, for any $a < b \in \R,$ with $F=\Atilde[t^{-}],$
  \[
    \left\{ u \in \R : F^{-1}(u) \in (a,b) \right\}
    \subseteq
    [F(a),F(b))
    \subseteq
    \left\{ u \in \R : F^{-1}(u) \in [a,b] \right\}
  \]
  From the continuity of $\Psi$ it follows that 
  \[
    \int \one\{u : F^{-1}(u) \in (a,b]\}
    \,d\Psi(u)
    =
    \Psi(F(b)) - \Psi(F(a))
    =
    \int \one\{ u \in (a,b]\}
    d\Psi(F(u)).
  \]
  By a monotone class argument, it follows for all non--negative Borel measurable functions $f$ and continuous $\Psi,$
  \[
    \int
    f(\Atilde[s^{-}]^{-1}(u))\,d\Psi(u) 
    =
    \int
    f(u)\,d\Psi(\Atilde[s^{-}](u)). 
  \]
  Thus we conclude that
    \[
    \int_0^1
    \int_0^1
    \tilde{B}^{\alpha}_s(u,v,x)\,dv 
    d\Psi(u)
    =
    x^2
    \int_x^\infty
    \frac{d\Aatilde[s^{-}]}{d\Atilde[s^{-}]}(z)
    \cdot
    \frac{d\Psi(\Atilde[s^{-}](z))}{z}
    \]
    As we then integrate the expression in time, we may freely replace the $s^{-}$ by $s$
    to conclude the proof.
\end{proof}

\noindent We also define $(\Aa,\Aaa,\A)$ and $\Ma$ by letting $\Aa[t](x) = \Aatilde[t](e^{t}x)$ for all $t,x \geq 0$ and similarily for $\Aaa,\A,$ and $\Ma.$  After this time change, we can write
\begin{equation}
  \label{eq:dynamics}
  \Aa[t](x)
  =
  \Aa[0](x)+ 
  \int_0^t (e^{s-t}x)^2 \int_{e^{s-t}x}^\infty 
  \frac{d\Aa[s]}{d\A[s]}(z)
  \cdot\frac{ d\Psi(\A[s](z))}{z}\,ds 
  + \Ma[t](x),
\end{equation}
for all $t,x \geq 0.$  This motivates the study of the operator on pairs of time evolving distribution functions that appears as the drift term in this decomposition.  To formalize this operator, we recall some notation from \cite{candy,Junge}.

Define the space $\Lloc$ of locally integrable functions $f:[0,\infty)\to\R$, endowed with the following canonical metric $\dblah$,
\[
\dblah(f,g) = \sum_{k=1}^\infty 2^{-k} \wedge \int_0^k |f(x)-g(x)|\,dx,
\]
which makes $\Lloc$ into a complete separable metric space. Define the subspace $\DSpace\subset\Lloc$ of subdistribution functions by
\[
\DSpace = \left\{
F : [0, \infty] \to [0,1],\text{ c\`{a}dl\`{a}g, increasing}
\right\}.
\]
In this paper we will reserve the term \emph{distribution function} for cumulative distribution functions of probability measures. Recall that convergence under $\dblah$ of distribution functions is equivalent to vague convergence:
\begin{lemma}[Lemma 2.3 of \cite{candy}]
\label{lem:L1loc_pointwise}
For $F,F_1,F_2,\ldots\in\DSpace$, $F_n\to F$ with respect to $\dblah$ if and only if $F_n(x)\to F(x)$ at every point of continuity $x$ of $F$. 
\end{lemma}
\noindent Lemma~\ref{lem:L1loc_pointwise} and Helly's selection theorem imply in particular that $(\DSpace,\dblah)$ is a compact metric space.

In \cite{candy,Junge}, the following norm played an important role.  Define for $f\in\Lloc$:
\[
\candy[f] = \int_0^\infty x^{-2}|f(x)|\,dx\in[0,\infty].
\]
Also define a subspace $\DSpace_1\subset\DSpace$ by
\[
\DSpace_1 = \left\{
F\in\DSpace: \candy[F] \leq 1
\right\}.
\]
Using $\dcandy,$ the space $\DSpace_1$ becomes a complete metric space.
\begin{lemma}[Lemma of \cite{candy}]
\label{lem:complete_candy}
The metric space $(\DSpace_1,\dcandy)$ is complete.
\end{lemma}

We now define the space  $\mathcal B([0,\infty),\Lloc)$ of Borel measurable maps from $[0,\infty)$ to $\Lloc$. 
  We endow this space with the topology of locally uniform convergence, which we denote by the symbol $\Dto$. Then $\F^{(n)} \Dto \F$ as $n\to\infty$ if and only if for all compact $K \subseteq [0,\infty)$ and all $t > 0,$
\[
\lim_{n \to \infty} \sup_{0 \leq s \leq t} \int\limits_K | \F[s]^{(n)}(x) - \F[s](x) |\,dx = 0.
\]
The subspaces $\DTSpace,\DTSpace_1\subset\mathcal B([0,\infty),\Lloc)$ are defined by
\[
\DTSpace = \mathcal B([0,\infty),\DSpace),\quad \DTSpace_1 = \mathcal B([0,\infty),\DSpace_1) \subset \DTSpace.
\]
Since $\DSpace$ and $\DSpace_1$ are closed subsets of $\Lloc$, $\DTSpace$ and $\DTSpace_1$ are closed subsets of $\mathcal B([0,\infty),\Lloc)$.%

The spaces of continuous maps $C([0,\infty),\DSpace)$ and $C([0,\infty),\Lloc)$ are closed subsets of $\DTSpace$ and $B([0,\infty),\Lloc)$, respectively. Furthermore, the topology on these spaces can be metrized to make them complete separable metric spaces.

For measures $\mu$ and $\nu$, write $\mu\ll \nu$ if $\mu$ is absolutely continuous w.r.t.~$\nu$. Define a space $\RTSpace$ by
\[
  \RTSpace = \left\{ (\F,\G) \in \DTSpace \times \DTSpace : d\G[t] \ll d\F[t]\,\forall t\ge0 \right\}.
\]
For \emph{absolutely continuous} $\Psi$ we define an operator
\(
\COP \colon \RTSpace \to \DTSpace
\)
by the formula
\begin{align}\label{eq:cop}
  \COP( \F,\G)_t(x)
  = {\G[0]}(e^{-t}x) + \int_0^t (e^{s-t}x)^2 \left[\int_{e^{s-t}x}^\infty 
    \frac{d\G[s]}{d\F[s]}(z)\cdot \frac{d(\Psi({\F[s]}(z)))}{z}
  \right]\,ds
\end{align}
for all $t,x \geq 0.$ Note that the integral over $z$ is well defined since the Radon-Nikodym derivative $d\G[s]/d\F[s]$ is defined $d\F[s]$-almost surely and the measure $d\Psi\circ \F[s]$ is absolutely continuous w.r.t.~$d\F[s]$ (see Remark~\ref{rem:abscontinuity}).
This allows us to rewrite \eqref{eq:dynamics} as 
\(
\Aa = \COP( \A, \Aa) + \Ma.
\)

We would like this operator to be continuous. This is not true in general as a map from $\RTSpace \to \DTSpace.$  However, the operator is continuous at certain specific points.
\begin{proposition}
  \label{prop:continuity}
Let $(\F^n,\G^n)$ be a sequence in $\RTSpace$ and let $(\F,\G)\in\RTSpace$. Suppose furthermore
\begin{enumerate}
\item $\Psi \in C^1[0,1],$
\item $\F^n\Dto \F$ and $\G^n\Dto \G$,
\item the map $(t,x) \mapsto \F[t](x)$ is continuous,
\item $\left\{ \F[t]: t \geq 0 \right\}$ is a tight family of distribution functions.   
\end{enumerate}
Then, $(\F,\G)\in\RTSpace$ and
\[
\COP(\F^n,\G^n) \Dto \COP(\F,\G).
\]
\end{proposition}
The proof of Proposition~\ref{prop:continuity} is postponed to Section~\ref{sec:proof_continuity}.

In the case that $\alpha = 1,$ \eqref{eq:dynamics} becomes 
\(
\A = \COP( \A, \A) + \M.
\)
In the notation of \cite{candy}, we write $\PDEOP : \DTSpace \to C( [0,\infty), \Lloc),$ which is given by $\PDEOP(\F) = \COP(\F,\F),$ so that $\A = \PDEOP(\A) + \M.$   
We also introduce the following family for $\PDEOP$:
\[
  \FPSpace = \{\F \in \DTSpace_1: \F = \PDEOP(\F), \forall t \geq 0: \F[t](+\infty) = 1 \text{ and } \{\F[t]\}_{t\geq 0} \text{ tight} \}.
\]
Here, we recall that a family of distribution functions $\{\F[\beta]\}_{\beta \in X}$ on $[0,\infty)$ is tight if for all $\epsilon > 0,$ there is an $N > 0$ sufficiently large such that, for every $\beta \in X,$ $\F[\beta](N) > 1-\epsilon.$

Of particular importance, as a consequence of Theorem~\ref{thm:candy} (more specifically, \cite[Lemma 3.5]{candy}), for all distributions $\Psi \in C^1[0,1],$ we have that there is a stationary fixed point $\F^* \in \FPSpace$ which has $F^*_t = F^{\Psi}$ for all $t \geq 0.$ Following the insight of \cite{Junge} and adapting his formalism, we additionally introduce the operator
\(
\COP^* \colon \DTSpace^* \to \DTSpace
\)
given by 
\begin{align*}
&\DTSpace^* = \{\F\in \DTSpace_1: d\F[t] \ll dF^\Psi\,\forall t\ge0,\,\sup_{t\ge0}\|\tfrac{d\F[t]}{dF^\Psi}\|_{\operatorname{L}^\infty(dF^\Psi)} < \infty\}\\
&\COP^*(\F) = \COP(\F^*,\F),\quad \F\in\DTSpace^*
\end{align*}
 and we introduce the family of fixed points
\[
  \CFPSpace = \{\F \in \DTSpace^*: \F = \COP^*(\F), \forall t \geq 0: \F[t](+\infty) = 1 \text{ and } \{\F[t]\}_{t\geq 0} \text{ tight} \}.
\]
The key contribution of this article is to show that in fact every solution in $\CFPSpace$ has the same $t \to \infty$ limit, $F^{\Psi},$ which we do by study of a related piecewise deterministic Markov process in Section~\ref{sec:growth-fragmentation}. The following proposition is a special case of Proposition~\ref{prop:convergence_RFPSpace}:
\begin{proposition}
  \label{prop:ergodicity}
  Let $\F \in \CFPSpace$. Then $d\F[t]$ converges in total variation to $dF^{\Psi}.$  Furthermore, if $\mathcal M$ is a tight family of probability distributions on $(0,\infty)$, the convergence is uniform on those $\F$ which furthermore satisfy $d\F[0] \in \mathcal M$.
\end{proposition}

It is now relatively simple to complete the proof of Theorem~\ref{thm:main}.  We need the following simple facts about the dynamics which are easy adaptations of arguments in \cite{candy,Junge}.

\begin{proposition}
\label{prop:evo_prop}
For $\Psi \in C^1[0,1],$ and any $\alpha \in (0,1],$ the following hold almost surely:
\begin{enumerate}[(i)]
  \item The collection of distribution functions $\{\alpha^{-1}\Aa[t]\}_{t \geq 0}$ is tight. \label{p:tight}
  \item The family $\{\alpha^{-1}\Aan\}_{n \geq 0}$ defined by $\Aan[t] = \Aa[t+n]$ for every $t\ge0$ is \emph{asymptotically equicontinuous}, i.e.\,for any $K$ compact
\[
\lim_{\delta \to 0}
\limsup_{t_0 \to \infty}
\sup_{\substack{s,t \geq t_0 \\ |s-t| \leq \delta} } 
\int\limits_K \left| \Aa[s](x) - \Aa[t](x) \right|\,dx=0.
\] \label{p:aec}
\item The noise vanishes in the limit, i.e.~$\Man \Dto 0$ as $n \to \infty$, where $\Man[t] = \Ma[t+n] - T_t \Ma[n]$ for every $t\ge0$, where $T_tf(x) = f(e^tx).$ \label{p:noise}
\end{enumerate}
\end{proposition}
\begin{proof}
  This is a simple modification of \cite[Proposition 7.2]{candy} or \cite[Proposition 5]{Junge}.
\end{proof}
With these properties in hand, it is now elementary to prove the following convergence:
\begin{proposition}
  \label{prop:candyass}
For each $\alpha \in [0,1],$ we have that 
\[
  \lim_{t \to \infty} \Aa[t] = \alpha F^{\Psi}
\]
almost surely.
\end{proposition}
\begin{proof}
Let $\alpha\in(0,1]$ (the case $\alpha=0$ is trivial).
Let $\left\{ t_k \right\}_{k=0}^\infty$ be an arbitrary sequence of times converging to $\infty$. It is enough to show that $\dblah(\alpha^{-1} \Aa[t_k],F^{\Psi})\to0$ along a subsequence. Fix  $\epsilon > 0$. Let $\mathcal{H}$ be the $\dblah$ closure of $\left\{ \alpha^{-1}\Aa[t] \right\}_{t \geq 0}.$  From Proposition~\ref{prop:evo_prop}(\ref{p:tight}) and the fact that for each $F \in \mathcal{H},$ 
  \[
    \candy[F] \leq \sup_{t \geq 0}\alpha^{-1}e^{-t}N_t < \infty,
  \]
  we have that $\mathcal{H}$ is tight as a family of distributions on $\left( 0,\infty \right).$  As a consequence,   Proposition~\ref{prop:ergodicity}  yields that there exists $T>0$, such that
  \begin{equation}
  \label{eq:suplate}
     \sup_{\F\in\CFPSpace: \F[0] \in \mathcal H}\dtv(d\F[T],\pi) < \epsilon.
  \end{equation}

Define $n_k = t_k - T$ for all $k \in\mathbb{N}$  (assume w.l.o.g.~that $t_k\ge T$ for all $k$).  For $s\ge0$, let $\Aas = (\Aas[t])_{t \geq 0}=(\Aa[t+s])_{t \geq 0}$ be the shifted process. By Proposition~\ref{prop:evo_prop}(\ref{p:aec}),  the family $\left\{ \Aank \right\}_{0}^\infty$  is asymptotically equicontinuous.  Hence by \cite[Lemma 7.3]{candy}, we may  extract a subsequence $\left\{ \Aankt \right\}_{0}^\infty$ which converges in $\DTSpace$ to $\alpha \F^{\alpha}$ for some $\F^{\alpha}$. In particular, with $ t'_k =  n'_k + T$, we have for all $k$ sufficiently large,
\begin{equation}
\label{eq:itslate}
\dblah(\alpha^{-1} \Aa[ t'_k], \F[T]^{\alpha}) =
\dblah(\alpha^{-1} \Aankt[T], \F[T]^{\alpha}) < \epsilon.
\end{equation}
Furthermore, we have 
  \[
    \alpha^{-1}\Aankt = \COP\left( 
    \Ankt, \alpha^{-1}\Aankt
    \right)
    + \Mankt \Dto \F^{\alpha}.
  \]
  
Since $\Aa[t]$ has Radon-Nikodym derivative bounded by 1 w.r.t.~$\A[t]$ for every $t\ge0$ and $\Ankt \to \F^*$ as $k\to\infty$, $\F[t]^\alpha$ has Radon-Nikodym derivative bounded by $\alpha^{-1}$ w.r.t.~$F^\Psi$ for every $t\ge0$. Hence, $\F^\alpha \in \DTSpace^*$ and $(\F^*,\F^\alpha)\in\mathcal R$.
  
  By Proposition~\ref{prop:evo_prop}(\ref{p:noise}), $\Mankt \Dto 0.$  By Theorem~\ref{thm:candy}, we have that $\A[t] \to F^{\Psi}$ almost surely, and hence $\Ankt \Dto \F^*,$ which is jointly continuous in time and space.  Then by Proposition~\ref{prop:continuity}, we have that 
  \[
    \F^{\alpha} = \COP(\F^*, \F^{\alpha}) = \COP^*(\F^{\alpha}).
  \]
By definition of the space $\mathcal H$, $\F[t]^\alpha \in \mathcal H$ for every $t\ge0$, in particular, the family $\{\F[t]^\alpha\}_{t\geq0}$ is a tight family of distribution functions of probability measures, and so $\F^\alpha\in \CFPSpace$. 
Hence, \eqref{eq:suplate} implies
\begin{equation}
\label{eq:itsverylate}
\dtv\left( dF_T^\alpha, \pi \right) < \epsilon.
\end{equation}

The $\dblah$ distance of two distribution functions on $[0,\infty)$ can easily be estimated in terms of the total variation distance of the underlying measures. By \eqref{eq:itslate} and \eqref{eq:itsverylate}, we then get that there is some monotone increasing function $f : \R^+ \to \R^+$ with $\lim_{x \to 0}f(x) = 0$ so that 
  \(
\dblah(\alpha^{-1} \Aa[t'_k],\pi) < \epsilon + f(\epsilon)
\).
As $\epsilon$ can be made arbitrarily small, the proof is complete. 
\end{proof}

\begin{proof}[Proof of Theorem~\ref{thm:main}]
Recall that we wish to show that $e^{-t}N_t^{\alpha}\to \alpha$ almost surely as $t\to\infty$, for every $\alpha\in[0,1]$. The key is the following chain of inequalities:
\begin{align}
\label{eq:sandwich}
  e^{-t}(N_t^{\alpha}-1) \leq \candy[\Aa[t]] \leq e^{-t}N_t^{\alpha},
\end{align}
which are easy consequences of the following identity for size-biased distribution functions \cite[Lemma 2.2]{candy}:
\[
  \candy[F] = \int_0^\infty x^{-1} dF(x),
\]
This identity furthermore implies $\candy[\A[t]] = e^{-t}N_t$. 

Fix $\alpha\in[0,1]$. Recall that $\candy[F^{\Psi}] = 1$ \cite[Lemma~3.5]{candy}. Together with Proposition \ref{prop:candyass} and Fatou's lemma, we have
  \[
    \liminf_{t\to\infty} \candy[\Aa[t]] \geq \alpha
  \]
  almost surely. By symmetry under reversing the interval $[0,1],$ it also follows that $\liminf_{t\to\infty} \candy[\Aaa[t]] \geq (1-\alpha).$  However, by the above, almost surely,
  \[
    \lim_{t \to \infty} \candy[\Aa[t]]+\candy[\Aaa[t]]
    =
    \lim_{t \to \infty} \candy[\A[t]] 
    =
    \lim_{t \to \infty} e^{-t}N_t  
    = 1,
  \]
  whence
  \[
    \lim_{t\to\infty} \candy[\Aa[t]] = \alpha
  \]
  almost surely. Equation~\eqref{eq:sandwich} now gives
  \[
      \lim_{t\to\infty} e^{-t}N_t^{\alpha} = \alpha,
\] 
which was to be proven.
\end{proof}

\section{The cell process}
\label{sec:growth-fragmentation}

This section can be read independently of the remainder of the article. Its goal is to prove Proposition~\ref{prop:ergodicity}, but in order to simplify notation and to make logical dependencies clearer we state a more general result, Proposition~\ref{prop:convergence_RFPSpace}. Throughout the section, we fix a non-negative function $R\in \Lloc.$  For the application to $\Psi$--processes, we will set $R = \psi \circ F^{\Psi}$, however, the results established in this section will not require any boundedness or continuity properties of the function $R$.
Define the operator $\COPR$ given by
\[
\COPR(\F)_t(x)
= {\F[0]}(e^{-t}x) + \int_0^t (e^{s-t}x)^2 \left[\int_{e^{s-t}x}^\infty \frac{R(z)}{z}\,d\F[s](z)\right]\,ds,
\]
for all $\F \in \DTSpace_{\COPR} = \{\F\in\DTSpace: d\F[s] \ll \mathrm{Leb}\}$, where $\mathrm{Leb}$ is Lebesgue measure on $\R_+$. In parallel with previous definitions, define 
\[
  \RFPSpace = \{\F \in \DTSpace_{\COPR}: \F = \COPR(\F), \forall t \geq 0: \F[t](+\infty) = 1,\ \F[t](0)=0 \text{ and }
  \{\F[t]\}_{t\geq 0} \text{ tight} \}.
\]

Throughout the section, we will work under the following assumption, which will be satisfied in the cases in which we are concerned.
\begin{assumption}\ 
  \label{ass:stat}
There is a stationary element $\F \in \RFPSpace,$ that is $\F[t] = F$ for some distribution function $F$ and all $t \geq 0.$ Then $F$ is the distribution function of a probability measure, which we denote by $\pi$.
\end{assumption}
\noindent This assumption is equivalent to \eqref{eq:ZR}, as we shall show below in Lemma \ref{lem:R}.

We also define a subspace of $\RFPSpace$ by
\[
  \RFPSpacea = \{\F \in \RFPSpace: \sup_{t \geq 0} \| \tfrac{d\F[t]}{d\pi}\|_{\operatorname{L}^\infty(\pi)} < \infty \}.
\]
These naturally arise in our application.  Our goal in this section is to show the following proposition:
\begin{proposition}
\label{prop:convergence_RFPSpace}
Under Assumption~\ref{ass:stat}, every $\F\in\RFPSpacea$ satisfies 
$$
\dtv(d\F[t],\pi)\to 0,\quad\text{as $t\to\infty$.}
$$
Furthermore, if $\mathcal M$ is a tight family of probability distributions on $(0,\infty)$, the convergence is uniform on the set $\{\F\in\RFPSpacea: d\F[0] \in \mathcal M\}$.
\end{proposition}

We now describe the global structure of the proof, which is split over several subsections. Let us first define the operator  $\RG: C_c^1(0,\infty) \to \Lpi$ by,
\begin{equation}
\label{eq:RG}
  \RG f(x) = xf'(x) + xR(x)\left[\int_0^x\frac{2u}{x^2}f(u)\,du - f(x)\right],
\end{equation}
where the fact that $\RG f \in \Lpi$ easily follows from \eqref{eq:integral_R}.
We start by showing in Section~\ref{sec:diff_mu} that for every $F\in \RFPSpace$, the family of measures $\mu_t = d\F[t]$ solves the equation
\begin{equation}
\label{eq:diff_mu}
	(\mu_t,f) = (\mu_0,f) + \int_0^t 
	\left( \mu_s,\RG f  \right)\,ds,\quad \forall f \in C_c^1(0,\infty),\ \forall t \geq 0,
\end{equation}
where $(\mu,f):= \int f\,d\mu$ for every $\mu$ and $f$.
This motivates Section~\ref{sec:construction_markov_process}, where we construct a piecewise deterministic Markov process, \emph{the cell process}, and identify $\RG$ as its infinitesimal generator (in a certain sense). This process can be seen as describing the evolution of the tagged fragment in a certain growth--fragmentation process (see \cite{Bertoin}), but we will not exploit this relation further. We also show in that section that $\pi$ is an invariant measure of the Markov process. 

The goal of the next section, Section~\ref{sec:Lpi_and_uniqueness}, is to show that every solution $(\mu_t)_{t\ge0}$ to \eqref{eq:diff_mu} satisfies $\mu_t = \mu_0 P_t$, i.e.~that $\mu_t$ is the law of the Markov process at time $t$, when the starting point is distributed according to $\mu_0$. The arguments required for this are of analytic nature. We first show that the transition operator of the Markov process defines a strongly continuous contraction semigroup $(P_t)_{t\ge0}$ on $\Lpi$. The main part is then to prove that the space $C_c^1(0,\infty)$ is a core of the domain of its generator. For this, we adapt ideas from the growth-fragmentation literature (especially the decomposition approach of \cite{Mischler2016}). This then allows to prove that every solution $(\mu_t)_{t\ge0}$ to \eqref{eq:diff_mu} is uniquely determined from its initial condition and indeed given by $\mu_t = \mu_0P_t$.

Having established that $\mu_t$ is the law of the Markov process at time $t$ starting from the law $\mu_0$, the last step is to show that this process is ergodic. This is done in Section~\ref{sec:ergodicity} using purely probabilistic arguments based on coupling of regenerative processes. Again, no additional assumptions on $R$ are necessary, ergodicity follows from the mere existence of a stationary probability, together with certain irreducibility properties. In particular, no Lyapunov functions are used.  With everything in place, Section~\ref{sec:proof_prop} wraps up the proof of Proposition~\ref{prop:convergence_RFPSpace}.

To finish this section, let us collect some consequences of Assumption~\ref{ass:stat} which will be of use later. First, under Assumption~\ref{ass:stat}, it is easy to check that $F$ is in fact continuously differentiable in $(0,\infty)$ and satisfies the identity
  \begin{equation}
  \label{eq:Fprime}
    \frac{F'(x)}{x} = \int_x^\infty R(z)\,\frac{F'(z)}{z}\,dz
  \end{equation}
  for  every $x > 0.$ Solving \eqref{eq:Fprime} for $F'(x)/x$ (see also the  proof of Corollary 9.3 in \cite{candy}) then shows that for every $0<x\le y<\infty$,
\begin{equation}
\label{eq:Fprime_exp}
\frac{F'(y)}{y} = \frac{F'(x)}{x} \exp\left(-\int_x^y R(z)\,dz\right).
\end{equation} 
   In particular, since $F'$ cannot be zero everywhere by the assumptions on $F$, $F'(x)> 0$ for all $x\in(0,\infty)$. In particular, $\pi$ is equivalent to Lebesgue measure on $(0,\infty)$.
 
\begin{lemma} \label{lem:R}
  Assumption \ref{ass:stat} holds if and only if \eqref{eq:ZR} holds.
\end{lemma}
\noindent We shall not use this Lemma in the development that follows. It is only provided to show that \ref{ass:stat} is a relatively simple condition.
\begin{proof}
  Taking $x=1$ in \eqref{eq:Fprime_exp} shows that
  \[
    F'(y) = y F'(1)\exp\left(-\int_1^y R(z)\,dz\right)
  \]
  Hence on integrating both sides over all $y$, we conclude that
  \[
    1 = F'(1) \int_0^\infty y \exp\left(-\int_1^y R(z)\,dz\right)\,dy
    = F'(1)Z^R.
  \]
  As $F'$ is everywhere positive, it follows that $Z^R < \infty.$

  On the other hand, if $Z^R < \infty,$ then we can define the probability density
  \[
    f(x) = x\exp\left(-\int_1^x R(y)\,dy\right)/Z^R, \quad \text{ for all } x > 0.
  \]
  Letting $F$ be its cumulative distribution function, it is elementary to see that $\F=(\F[t])_{t\geq 0}$ with $\F[t] = F$ for all $t \geq 0$ is a stationary element of $\RFPSpace.$
\end{proof}

Also note that the positivity of $F'$ together with \eqref{eq:Fprime} implies the following:
\begin{equation}
\label{eq:R_positivity}
\forall x\in(0,\infty): \int_x^\infty R(y)\,dy > 0.
\end{equation}
Actually, the last integral can be shown to be infinite, but we will not use this fact explicitly.
 Furthermore, multiplying both sides in \eqref{eq:Fprime} by $x$, then integrating over all $x \geq 0$ and applying Fubini's theorem, we can additionally conclude that
  \begin{equation}
  \label{eq:integral_R}
    \int_0^\infty z R(z)\,\pi(dz) = 2\int_0^\infty \,\pi(dx) = 2.
  \end{equation}

\subsection{Proof of (\ref{eq:diff_mu})}
\label{sec:diff_mu}

The goal of this section is to prove the following proposition:
\begin{proposition}
  \label{prop:markov}
  For every $\F \in \RFPSpace,$ $f \in C_c^1(0,\infty)$ and all $t \geq 0,$ the following holds:
      \[
	(d\F[t],f) = (d\F[0],f) + \int_0^t 
	\left( d\F[s],\RG f  \right)\,ds,
      \]
i.e.~the family of measures $(d\F[t])_{t\ge0}$ solves \eqref{eq:diff_mu}.      
\end{proposition}
\begin{proof}

Throughout the proof, fix $\F \in \RFPSpace$ and $f \in C_c^1(0,\infty)$. Define the semigroup $(T_t)_{t\ge0}$ by $T_tf(x) = f(e^t x)$. Furthermore, define the operator 
\[
  \RGT f(x) = \RG f(x) - xf'(x) = xR(x)\left[\int_0^x\frac{2u}{x^2}f(u)\,du  - f(x)\right].
\]
We now first show the following fact:
For all $t \geq 0,$
\begin{equation}
\label{eq:diff_aux}
	(d\F[t], f) = (d\F[0],T_tf) + \int_0^t 
	\left( d\F[s],\RGT T_{t-s}f \right)\,ds,
\end{equation}
This equation will follow directly from the definition of $\RFPSpace$.  Using the boundary conditions on $f,$ we can write
      \[
(d\F[t], f)
=\int_0^\infty f(x)\,dF_t(x)
=-\int_0^\infty f'(x) F_t(x)\, dx.
      \]
      We may now use the definition of $\COPR$ to write
      \[
(d\F[t], f)
=
-\int\limits_0^\infty F_0(e^{-t}x) f'(x)\,dx
-\int\limits_0^\infty f'(x)\int\limits_0^t (e^{s-t}x)^2 \left[\int_{e^{s-t}x}^\infty \frac{R(z)}{z}\,d\F[s]\right]\,ds
\,dx.
      \]
      The first term we integrate by parts and change variables.  To the second term, we apply Fubini to bring the $t$ integral to the outside and the $x$ integral to the inside (the justification for the change of order of integration follows readily from $\F = \COPR(\F)$).  This gives
      \[
(d\F[t], f)
=
\int\limits_0^\infty f(e^{t}x)\,dF_0(x)
-
\int\limits_0^t
\int\limits_0^\infty
  \frac{R(z)}{z}
  \int_0^{e^{t-s}z}
f'(x) (e^{s-t}x)^2 
\,dx
d\F[s](z)
ds.
      \]
      Finally, integrating the $x$ integral by parts, we arrive at
      \[
(d\F[t], f)
=
\int\limits_0^\infty f(e^{t}x)\,dF_0(x)
+
\int\limits_0^t
\int\limits_0^\infty
  \frac{R(z)}{z}
  \biggl[
  \int_0^{z}
  2xf(e^{t-s}x)\,dx
  -z^2f(e^{t-s}z)\biggr]
d\F[s](z)
ds,
      \]
      which is \eqref{eq:diff_aux}.  

We now claim that the map $t \mapsto (d\F[t], f)$ is continuously differentiable. To see this, we begin by noting that the equation $t \mapsto (d\F[t],f)$ is continuous by \eqref{eq:diff_aux}.  Moreover, for fixed $f,$ the function
      \[
	H(s,t) = (d\F[s],\RGT T_{t-s}f)
      \]
      is continuous in $s$ and $t,$ and in fact differentiable in $t$ as can be seen by differentiating under the integral sign in $(d\F[s],\RGT T_{t-s}f).$  Specifically, we have that 
    \[
    \frac{|\RGT[T_{b}f](x)-\RGT[T_{b}f](x)|}{|b-a|}
    \leq xR(x)C_f,
    \]
    for all $a,b \in (0,\infty)$ some constant $C_f$ depending only on $f.$  Thus, the differentiation is justified, and we have that 
    \[
      \partial_t H(s,t) = (d\F[s],\RGT \partial_t(T_{t-s}f))
      =(d\F[s],\RGT T_{t-s}[xf']),
    \]
    which is continuous in $s.$  Hence, we get from \eqref{eq:diff_aux} that $t \mapsto (d\F[t], f)$ is continuously differentiable and moreover 
    \begin{align*}
      \partial_t (d\F[t], f) 
      &= (d\F[0],T_t[xf']) + 
      \int_0^t 
      \left( d\F[s],\RGT T_{t-s}[xf'] \right)\,ds
      +
      \left( d\F[t],\RGT f \right) \\
      \intertext{The first two terms can be combined by \eqref{eq:diff_aux}, however, to give:}
      \partial_t (d\F[t], f)
      &=(d\F[t], xf')+( d\F[t],\RGT f ) \\
      &=\left(d\F[t], \RG f \right).
    \end{align*}
    Hence, on integrating this expression, we arrive at the statement of the proposition.
\end{proof}

\subsection{Construction of the cell process}
\label{sec:construction_markov_process}
Corresponding to the operator $\RG,$ we construct a \emph{piecewise deterministic Markov process (PDMP)} $X = (X_t)_{t\ge0}$ on $(0,\infty),$ which we refer to as the \emph{cell process}. A PDMP is a Markov process that almost surely has a finite number of jumps in any finite time interval and which moves deterministically between the jumps. In our case, denote by $(x,t)\mapsto \Phi(x,t)$ the flow generated by the vector field $x\mapsto x\frac{\partial}{\partial x}$. Then $\partial_t \Phi(x,t) = \Phi(x,t)$ and $\Phi(x,0) = x$, so that $\Phi(x,t) = xe^{t}$.
Simply stated, the cell process flows from $x$ along the integral curve $xe^{t}$ until some random time (regulated by the \emph{jump rate} $r(x) = xR(x)$) at which point the process jumps downwards by a multiplicative factor whose distribution is the size bias of $\Unif[0,1].$  This procedure then restarts, flowing upwards and jumping downwards.

 It will be useful to define the process precisely, making use of the special features of the process which allow for a simple construction. The state space of our Markov process will be the interval $[0,\infty)$, the point $0$ being an absorbing state. The process will be defined on a probability space $(\Omega,\Pr)$ supporting an iid sequence $(U_1,U_2,\ldots)$ of r.v.~uniformly distributed in $(0,1)$. The process will be denoted by $X^x = (X^x_t)_{t\ge0}$, where $x\in [0,\infty)$ is the starting point. The definition goes as follows.
 
If the starting point is $x = 0$, we set $X^x_t = 0$ for all $t\ge0$. Suppose now that $x\in (0,\infty)$.
Define for $z\in(0,\infty)$ the \emph{survivor function} $\surv_z(t)$ given by 
  \[
    \surv_z(t) = \exp\left( - \int_0^t T_s[r](z) ds \right) = \exp\left(- \int_z^{ze^t} R(u)\,du\right).
  \]
  Note that $\surv_z(t) > 0$ for all $z\in(0,\infty)$ and $t\ge0$, by the local integrability of $R$.
  
We construct a sequence of pairs $(\tau_k,J_k)_{k=1}^\infty$, where $\tau_k$ will be the time of the $k$-th jump, and $J_k$ the multiplicative factor of the jump of the process.
 These random variables also depend on $x$, but we suppress this from the notation.

The construction is done recursively: Set $\tau_0 := 0$ and $Y_0=0$. Let $k\in\N=\{1,2,\ldots,\}$. We define $(\tau_k,J_k,Y_k)$ as follows:
\begin{itemize}
\item If $\tau_{k-1} = \infty$, then $\tau_k = \infty$ as well, otherwise 
\[
\tau_k = \tau_{k-1} + \surv_{Y_{k-1}}^{-1}(U_{2k-1}),
\]
where $Y_k = x e^{\tau_k} \prod_{j=1}^k J_j$ and $ \surv_{x}^{-1}$ is the generalized inverse of the function $\surv_x$. Note in particular that since $J_j > 0$ for all $j$ and $\surv_{x}(t) > 0$ for all $x>0$ and $t\ge0$, we have $\tau_k > \tau_{k-1}$ for all $k$.
\item $J_k = \sqrt{U_{2k}},$ i.e.~$J_k$ takes values in $(0,1)$ and its law has density $u\mapsto 2u$ on $(0,1)$. We can and will assume that $\prod_{k=1}^\infty J_k = 0$ for all $\omega\in \Omega$.
\end{itemize}

We then define the process $X^x = (X^x_t)_{t\ge0}$ by
\begin{equation}
\label{eq:def_X}
X^x_t = 
\begin{cases}
xe^t\prod_{j=1}^k J_j = e^{t-\tau_k} Y_k & \text{if  $\tau_k \le t < \tau_{k+1}$ for some $k\in\N_0$}\\
0 & \text{if $t\ge \zeta \coloneqq \lim_{k\to\infty} \tau_k$}.
\end{cases}
\end{equation}
Note that in particular, $X^x_{\tau_k} = Y_k$ for all $k\in\N_0$. 
Finally, since $J_j >0$ for all $j$, we note the equivalence (if the starting point $x \ne 0$):
\begin{equation}
\label{eq:equivalence}
X^x_t = 0 \quad\text{iff}\quad t\ge \zeta.
\end{equation}

\emph{From now on, following common usage, we will rather work with a single stochastic process $X = (X_t)_{t\ge0}$ defined on a measurable space $(\Omega,\mathcal A)$, endowed with a family of probability measures $(\Pr_x)_{x\ge0}$, such that under $\Pr_x$, $X$ has the law of $X^x$. }We note that $(\Pr_x)_{x\ge0}$ is a probability kernel, i.e. it is measurable in $x$, since $X^x$ is a measurable function of the random variables $(\tau_k,J_k)_{k=1}^\infty$, which are easily seen by induction to be measurable w.r.t.~$x$; the measurability of $\Pr_x$ then is a consequence of (part of) Fubini's theorem.

Denote by $\filt = (\filt_t)_{t\ge0}$ the natural filtration of the process $X$. Further, set for all $t,x\ge0$ and every bounded Borel function $f:[0,\infty)\to\R$:
\begin{equation}
\label{eq:P_t}
P_tf(x) \coloneqq \E_x[f(X_t)].
\end{equation}
Finally, let $\Co$ be the space of bounded continuous functions on $[0,\infty)$ that are $0$ at $0,$ made into a Banach space using the supremum norm.  

\begin{proposition}
\label{prop:X_Markov}
\ 
\begin{enumerate}[(i)]
\item
The process $X$ is a homogeneous strong Markov process on $[0,\infty)$ with paths in the Skorohod space $D([0,\infty))$ of c\`{a}dl\`{a}g functions on $[0,\infty)$ and whose semigroup of transition operators is  $(P_t)_{t\ge0}$. In other words, for every $x\ge0$, every $\filt$-stopping time $\tau$ and every bounded Borel function $f:[0,\infty)\to\R$, we have
\[
\E_x[f(X_{\tau+t})\,|\,\filt_\tau] \1_{\tau<\infty} = P_tf(X_\tau)\1_{\tau<\infty}. 
\]
\end{enumerate}

Moreover, under Assumption~\ref{ass:stat}, we have the following:
\begin{enumerate}[(i)]
\setcounter{enumi}{1}
\item The probability measure $\pi$ is invariant for the Markov process $X$.
\item For every $f\in C_c^1(0,\infty)$, we have 
$$
\lim_{t\to\infty} t^{-1}(P_t f - f) = \RG f,\quad\text{in $\Lpi$}.
$$ 
\end{enumerate}
\end{proposition}
\begin{proof}
Point (i) is proven in \cite[Chapter 25]{Davis1993} for general PDMP under the assumption that for every $x\in [0,\infty)$, the expected number of jumps of the process started from $x$ is finite in every finite time interval. This assumption is 
difficult to verify \emph{a priori} (and possibly false), as the jump rate may not be bounded near 0. We therefore localize the construction, i.e.~we construct a sequence of approximating processes through an appropriate family of stopping times which satisfy this assumption. 

Assume $x\in(0,\infty)$, the case $x=0$ being trivial. Then define for every $n\in\N$, $n> 1/x$,
\[
T^n = \inf\{t: X_t \le 1/n\},\quad\text{and}\quad X^n_t = \begin{cases} X_t & \text{if $t<T^n$}\\ 0 & \text{otherwise}\end{cases}.
\]
Since the process can decrease only by jumps, we have the following fact:
\begin{equation}
\label{eq:Tn_tauk}
\text{If $T^n < \infty$, then there exists $k\in\N$, such that $T^n = \tau_k$.}
\end{equation}
Inspecting the construction of a PDMP in \cite[Chapter 24]{Davis1993}, we can then identify the process $X^n$ as a PDMP on $\{0\}\cup(1/n,\infty)$ which informally evolves as follows:
\begin{itemize}
\item In $[1/n,\infty)$, it moves according to the flow $\Phi$,
\item jumps at rate $r$,
\item according to the jump kernel $Q(x,A) = \Pr(xJ_1\1_{xJ_1 \ge 1/n} \in A)$, where $J_1$ is defined as above.
\item The state $0$ is an absorbing state, i.e. the process does not move and the jump rate is 0.
\end{itemize}
To be precise, in \cite{Davis1993}, the state space of the PDMP is supposed to be an open subset of $\R^d$, but this can be remedied by considering the enlarged state space $(-\ep,\ep)\cup(1/n,\infty)$ for small enough $\ep>0$, and make all states in $(-\ep,\ep)$ absorbing. In Lemma~\ref{lem:existence_n} below, we show that the process  $X^n$ satisfies the assumption of finite expected number of jumps from \cite{Davis1993}.
Theorem~25.5 in \cite{Davis1993} then states that $X^n$ is indeed a homogeneous strong Markov process (with c\`{a}dl\`{a}g paths), i.e.\ if we define $P^n_t f(x) = \E_x[f(X^n_t)]$ and denote by $\filt^n = (\filt^n_t)_{t\ge0}$ the natural filtration of $X^n$, then for every $\filt^n$-stopping time $\tau^n$,
\begin{equation*}
\E_x[f(X^n_{\tau^n+t})\,|\,\filt^n_{\tau^n}] \1_{\tau^n<\infty} = P_tf(X^n_{\tau^n})\1_{\tau^n<\infty}. 
\end{equation*}
Note that if $\tau$ is a $\filt$-stopping time, then $\tau\wedge T^n$ is a $\filt^n$-stopping time. Furthermore, we have $\filt^n_{\tau\wedge T^n} = \filt_{\tau\wedge T^n}$. Finally, observe that $X^n$ depends on $\filt_\tau$ only through $\filt_{\tau\wedge T^n}$, since $X^n_t = 0$ for $t\ge T^n$. Combining these observations gives for every $\filt$-stopping time $\tau$,
\begin{equation}
\label{eq:markov_Xn}
\E_x[f(X^n_{\tau\wedge T^n+t})\,|\,\filt_\tau] \1_{\tau\wedge T^n<\infty} = P_tf(X^n_{\tau\wedge T^n})\1_{\tau\wedge T^n<\infty}. 
\end{equation}
We now decompose \eqref{eq:markov_Xn} according to whether $\tau < T^n$ or $\tau \ge T^n$. The left-hand side then equals,
\begin{align*}
\E_x[f(X_{\tau+t})\,|\,\filt_\tau] \1_{\tau < T^n} + f(0) \1_{T^n \le \tau < \infty},
\end{align*}
and the right-hand side equals 
\begin{align*}
P_tf(X_\tau) \1_{\tau < T^n}  + f(0) \1_{T^n \le \tau < \infty},
\end{align*}
Passing to the limit $n\to\infty$ then gives
\begin{equation}
\label{eq:markov_X_pre}
\E_x[f(X_{\tau+t})\,|\,\filt_\tau] \1_{\tau < \lim_{n\to\infty} T^n} = P_tf(X_\tau) \1_{\tau < \lim_{n\to\infty} T^n} 
\end{equation}
We claim that $X_t = 0$ for all $t\ge\lim_{n\to\infty}T^n$. Indeed, by \eqref{eq:Tn_tauk}, $(T^n)_{n\ge1}$ is a subsequence of $(\tau_k)_{k\ge1}$, hence 
$$\lim_{n\to\infty} T^n = \lim_{k\to\infty} \tau_k = \zeta.$$
The claim now follows from \eqref{eq:equivalence}.
Hence, adding $f(0) \1_{\zeta \le \tau < \infty}$ to both sides of \eqref{eq:markov_X_pre} yields
\[
\E_x[f(X_{\tau+t})\,|\,\filt_\tau] \1_{\tau < \infty}= P_tf(X_\tau) \1_{\tau < \infty},
\]
which was to be proven.

It remains to show that $X$ is c\`{a}dl\`{a}g. Looking at the definition of $X$ in \eqref{eq:def_X}, this is obvious on $[0,\zeta)$, since only finitely many jumps occur in every interval $[0,a]$ for $a<\zeta$. Since $X_t = 0$ for $t\ge \zeta$ (by \eqref{eq:equivalence}), it remains to show that $X$ has a left-hand limit at $0$. But since $\tau_k$ is strictly increasing in $k$ and the function $t\mapsto e^t$ is continuous, we have
\[
\lim_{t\uparrow \zeta} X_t = \lim_{k\to\infty} X_{\tau_k} = xe^\zeta \prod_{j=1}^\infty J_j = 0 = X_\zeta,
\]
so that $X$ is in fact continuous at $\zeta$. This finishes the proof of point (i).

We now prove point (ii), i.e.~that $\pi$ is an invariant measure for the process $X$. For this, it is enough to show that there exists a function $\eta:\R_+\to \R_+$ with $\eta(t) \to 0$ as $t\to 0$ and such that for every Borel function $g:[0,\infty)\to [0,1]$,
\begin{equation}
\label{eq:to_show}
\forall t>0: \left|\int_0^\infty (P_t g(x) - g(x))\, \pi(dx)\right| \le t \eta(t).
\end{equation}
Indeed, if \eqref{eq:to_show} holds, then telescoping and using the semigroup property yields for every $t>0$, every $g$ as above and every $n\ge 1$,
\begin{align*}
\left|\int_0^\infty (P_t g(x) - g(x))\, \pi(dx)\right|
&= \sum_{k=0}^{n-1} \left|\int_0^\infty (P_{t/n}P_{kt/n} g(x) - P_{kt/n} g(x))\, \pi(dx) \right|
\end{align*}
Note that $P_s g$ takes values in $[0,1]$ for all $s\ge 0$ if $g$ does so. Applying \eqref{eq:to_show} then shows that
\[
\left|\int_0^\infty (P_t g(x) - g(x))\, \pi(dx)\right| \le n \frac t n \eta\left(\frac t n\right)
\]
and letting $n\to\infty$ shows that
\[
\int_0^\infty P_t g(x)\, \pi(dx) =  \int_0^\infty g(x)\, \pi(dx),
\]
where $g$ was any Borel function with values in $[0,1]$. This implies that $\pi$ is an invariant measure.

We now show \eqref{eq:to_show}. Fix a Borel function $g:[0,\infty)\to [0,1]$. We let $X' = (X'_t)_{t\ge0}$ be the process that equals $(X_t)_{t \geq 0}$ up to and including the first jump and then grows exponentially afterwards.  Formally,
\begin{equation}
\label{eq:Xprime}
X'_t = 
\begin{cases}
X_0 e^t,\quad t <\tau_1\\
X_0 e^t J_1,\quad t \ge \tau_1
\end{cases},
\end{equation}
and note that $X'_t = X_t$ for all $t < \tau_2$. Hence, for every $x\in (0,\infty)$ and $t>0$,
\begin{equation}
\label{eq:comparison_X_Xprime}
\big|\E_x[g(X_t)] - \E_x[g(X'_t)]\big| \le \Pr_x(\tau_2 \le t),
\end{equation}
using that $g$ takes values in $[0,1]$.
We now claim: 
\begin{align}
\label{eq:gen1} \forall t>0: \left|\int_0^\infty [\E_x[g(X'_t)] - g(x)]\,\pi(dx)\right| &\le 8t^2\\
\label{eq:gen2} \text{and}\qquad \Pr_\pi(\tau_2 \le t) /t &\to 0,\quad t\to0.
\end{align}
Then \eqref{eq:gen1} and \eqref{eq:gen2} together with \eqref{eq:comparison_X_Xprime} readily imply \eqref{eq:to_show}.

We prove \eqref{eq:gen1} and \eqref{eq:gen2} by direct calculation. Recall that $\pi(dx) = F'(x)\,dx$, with $F$ from Assumption~\ref{ass:stat}. Recalling the definition of the survival function $\surv_x$, we can rewrite \eqref{eq:Fprime_exp} as
\begin{equation}
\label{eq:xes}
\forall s\ge 0,\,x\in(0,\infty): F'(x) \surv_x(s) = e^{-s} F'(xe^s),
\end{equation}
and recalling that $\pi(dx) = F'(x)\,dx$, this gives for every bounded Borel function $h:(0,\infty)\to\R$,
\begin{equation}
\label{eq:xes2}
\forall s\ge 0: \int_0^\infty h(xe^s) S_x(s)\,\pi(dx) = e^{-2s} \int_0^\infty h(x)\,\pi(dx),
\end{equation}
where we applied first \eqref{eq:xes}, then a change of variables $xe^s\mapsto x$. This formula will be used in several places.

We now show \eqref{eq:gen1}. We decompose
\begin{align}
\E_x[g(X'_t)] 
&= \E_x[g(xe^t)\1_{t < \tau_1}] + \E_x[g(xe^tJ_1)\1_{t\ge \tau_1}]\nonumber\\
&=  g(xe^t)S_x(t) + \E[g(xe^tJ_1)] (1-S_x(t)) \label{eq:decomp},
\end{align}
where we used the fact that $\tau_1$ and $J_1$ are independent. Integrating w.r.t.~$\pi$, we have for the first term by \eqref{eq:xes2},
\begin{align}
\int_0^\infty g(xe^t)S_x(t)\,\pi(dx) = e^{-2t}  \int_0^\infty g(x) \,\pi(dx). \label{eq:decomp1}
\end{align}
As for the second term on the RHS of \eqref{eq:decomp}, we have again by \eqref{eq:xes2},
\begin{align}
\int_0^\infty \E[g(xe^tJ_1)] (1-S_x(t))\,\pi(dx) &= \int_0^\infty (\E[g(xe^tJ_1)] - e^{-2t} \E[g(xJ_1)])\,\pi(dx). \label{eq:decomp2}
\end{align}
Now, by the definition of the random variable $J_1$, we have for every $x\in(0,\infty)$,
\begin{align}
\E[g(xe^tJ_1)] - e^{-2t} \E[g(xJ_1)] 
&= \int_0^{e^t} \frac{2u}{(e^t)^2} g(xu)\,du - e^{-2t}\int_0^x 2u  g(xu)\,du \nonumber \\
&= e^{-2t} \int_1^{e^t} 2u \,g(xu)\,du \nonumber\\
&= 2 e^{-2t} \int_0^t e^{2s} g(xe^s)\,ds. \label{eq:decomp3}
\end{align}
Integrating against $\pi$ and using Fubini's theorem, we get
\begin{align}
\int_0^\infty \int_0^t e^{2s} g(xe^s)\,ds\,\pi(dx) 
&= \int_0^t e^{2s} \int_0^\infty g(xe^s)\,\pi(dx) \,ds \nonumber\\
&= \int_0^t \int_0^\infty g(x)\frac 1 {S_{xe^{-s}}(s)}\,\pi(dx). \label{eq:decomp4}
\end{align}
Since $S_{xe^{-s}}(s)$ is decreasing in $s$ and equals 1 at 0, we get from \eqref{eq:decomp4}, using that $g$ takes values in $[0,1]$,
\begin{align}
\left|\int_0^\infty \int_0^t e^{2s} g(xe^s)\,ds\,\pi(dx) - t \int_0^\infty g(x)\,\pi(dx)\right|
& \le  t \int_0^\infty \left(\frac 1 {S_{xe^{-t}}(t)}-1\right)\,\pi(dx)\nonumber\\
& = t (e^{2t}-1) \label{eq:decomp5},
\end{align}
where the last equality follows from \eqref{eq:xes2} applied to $h(x) = 1/S_{xe^{-t}}(t)$ and the fact that $\pi$ is a probability measure. Collecting  \eqref{eq:decomp2}, \eqref{eq:decomp3} and \eqref{eq:decomp5}, we get
\begin{align}
\left|\int_0^\infty \E[g(xe^tJ_1)] (1-S_x(t))\,\pi(dx) - 2t e^{-2t} \int_0^\infty g(x)\,\pi(dx)\right| \le 2t(1-e^{-2t}) \label{eq:decomp6}
\end{align}
Equations \eqref{eq:decomp}, \eqref{eq:decomp1} and \eqref{eq:decomp6} now yield, using that $g$ takes values in $[0,1]$,
\begin{align*}
\left|\int_0^\infty [\E_x[g(X'_t)] - g(x)]\,\pi(dx)\right| \le |(1+2t)e^{-2t} - 1|+ 2t(1-e^{-2t}) \le 8t^2
\end{align*}
using standard estimates on the exponential function, in particular $1 \ge (1+x)e^{-x}\ge 1-x^2$. This is \eqref{eq:gen1}.

It remains to prove \eqref{eq:gen2}.  Recall that $r(x) = xR(x)$ for all $x > 0$ and $\frac{d}{dt} S_x(t) = r(xe^t)S_X(t).$   By definition of the process $X$, we have for every $x\in(0,\infty)$,
\begin{align*}
\Pr_x(\tau_2 \le t) = \int_0^t r(xe^s)S_x(s) \E[1-S_{xe^sJ_1}(s)]\,ds,
\end{align*}
where the expectation is meant with respect to $J_1$.
Integrating against $\pi$ and using \eqref{eq:xes2}, we get
\begin{align}
\Pr_\pi(\tau_2 \le t) 
&= \int_0^t \left(\int_0^\infty r(xe^s)S_x(s)\E[1-S_{xe^sJ_1}(s)]\,\pi(dx)\right)ds\nonumber\\
&=  \int_0^t e^{-2s} \left(\int_0^\infty r(x)\E[1-S_{xJ_1}(s)]\,\pi(dx)\right)ds\nonumber\\
&\le t \int_0^\infty r(x)\E[1-S_{xJ_1}(t)]\,\pi(dx), \label{eq:pr_tau2}
\end{align}
since the function $S_y(s)$ is decreasing in $s$ for every $y\in(0,\infty)$.
Now, $\E[1-S_{xJ_1}(t)] \to 0$ as $t\to\infty$, for every $x\in (0,\infty)$. Furthermore by \eqref{eq:integral_R}, 
\[
\int_0^\infty r(x)\,\pi(dx) = 2.
\]
By dominated convergence, we thus get
\begin{equation}
\label{eq:pr_tau2_2}
\int_0^\infty r(x)\E[1-S_{xJ_1}(t)]\,\pi(dx) \to 0,\quad t\to 0.
\end{equation}
Equation \eqref{eq:gen2} now follows from \eqref{eq:pr_tau2} and \eqref{eq:pr_tau2_2}.

We now prove point (iii). Fix a function $f\in C_c^1(0,\infty)$. We want to show that $t^{-1} (P_t f - f)$ converges to $\RG f$ in $\Lpi$. Recalling the definition of $(X_t')_{t\ge0}$ from above, it is enough by \eqref{eq:gen2} to show the following convergence:
\begin{equation}
\label{eq:conv_Lpi}
\frac{\E_x[f(X_t')] - f(x)}{t} \to \RG f(x),\quad\text{in $\Lpi$}.
\end{equation}
By \eqref{eq:decomp}, we can decompose
\begin{align}
\label{eq:decomp10}
\frac{\E_x[f(X_t')] - f(x)}{t} = \frac{f(xe^t) - f(x)}t S_x(t) + (\E[f(xe^tJ_1)] - f(x)) \frac{1-S_x(t)}{t}.
\end{align}
Since $f'\in C_c(0,\infty)$ and $S_x(t) \uparrow 1$ as $t\to0$, the first term on the right-hand side of \eqref{eq:decomp10} converges to $xf'(x)$ in $\mathrm L^\infty$ as $t\to0$ and therefore in $\Lpi$. As for the second term, first observe that $\E[f(xe^tJ_1)] - f(x)$ converges to $\E[f(xJ_1)]-f(x)$ in $\mathrm L^\infty$ as $t\to0$, by the continuity of $f$. We claim that $(1-S_x(t))/t$ converges to $xR(x)$ in $\Lpi$. First, note that by the Lebesgue differentiation theorem the convergence holds Lebesgue-a.e., hence $\pi$-a.e. by the equivalence of the two measures. Second, we have convergence of the $\Lpi$-norms:
\begin{align*}
\int \frac{1-S_x(t)}{t}\,\pi(dx) 
&= \frac{1 - e^{-2t}} t \int \pi(dx) && \text{by \eqref{eq:xes2}}\\
&\to 2 && \text{as $t\to 0$}\\
&= \int x R(x)\,\pi(dx) && \text{by \eqref{eq:integral_R}}.
\end{align*}
By a theorem of Riesz known amongst probabilists as Scheff\'e's lemma (\cite{Kusolitsch}!), using the positivity of $1-S_x(t)$, we then get that $(1-S_x(t))/t$ converges indeed to $xR(x)$ in $\Lpi$. To summarize, we have proven the following convergence in $\Lpi$ as $t\to0$:
\[
\frac{\E_x[f(X_t')] - f(x)}{t} \to xf'(x) + xR(x) (\E[f(xJ_1)]-f(x)) = \RG f(x),
\]
where the last equality follows from the definition of $J_1$. This is exactly \eqref{eq:conv_Lpi} which was to be shown.
\end{proof}

The following lemma appeared in the above proof of Proposition~\ref{prop:X_Markov}:

\begin{lemma}
\label{lem:existence_n}
For every $t\ge0$ and every $x>1/n$, we have 
$$\E_x\left[\sum_{k=1}^\infty \1_{\tau_k \le t\wedge T^n}\right] < \infty.$$
\end{lemma}
\begin{proof}
Fix $t_*\ge0$. By \eqref{eq:def_X}, we have $X_t \le xe^t \le xe^{t_*}$ for all $t\le t_*$. In particular, for each $k\in\N$, conditioned on $\tau_1,\ldots,\tau_k$ and $J_1,\ldots,J_k$, on the event $\tau_k \le t_*\wedge T^n$, the difference $\tau^n_{k+1}-\tau^n_k$ is stochastically dominated from below by a random variable with survivor function
\[
E^n(t)
= \exp\left(- \sup_{x \in[1/n,xe^{t_*}]} \int_x^{e^tx} R(u)\,du\right) \quad \text{for all } 0 \leq t \leq t_*,
\]
and such a random variable is non-degenerate since $E^n(t) > 0$ for all $t$ by the local integrability of $R$. Hence, the sum $\sum_{k=1}^\infty \1_{\tau_k \le t\wedge T^n}$ is stochastically bounded from above by the number of points in the interval $[0,t_*]$ of a renewal process with interarrival times distributed according to the survivor function $E^n$. But this quantity has finite expectation, which finishes the proof.
\end{proof}

\subsection{Uniqueness of solutions to (\ref{eq:diff_mu})}
\label{sec:Lpi_and_uniqueness}
{\emph{Everywhere in this section, we work under Assumption~\ref{ass:stat}.}}

The goal of this section is to prove the following result:

\begin{proposition}
\label{prop:uniqueness}
Suppose that $(\mu_t)_{t \geq 0}$ is a family of Borel probability measures absolutely continuous with respect to $\pi$ satisfying that $t \mapsto \mu_t$ is Borel measurable and so that
  \[
    \sup_{t \geq 0} \left\|\tfrac{d\mu_t}{d\pi}\right\|_{\operatorname{L}^\infty(\pi)} < \infty.
  \]
  Suppose further that $(\mu_t)_{t\ge0}$ is a solution to \eqref{eq:diff_mu}, i.e., for all $f \in C_c^1(0,\infty)$ and $t\ge0$,
  \[
    (\mu_t, f) - (\mu_0,f) = \int_0^t (\mu_s, \RG f)\,ds.
  \]
  Then $\mu_t = \mu_0 P_t$ for all $t\ge0$.
\end{proposition}

As mentioned above, the proof of Proposition~\ref{prop:uniqueness} will heavily rely on semigroup theory. The results from the previous section will be an important ingredient as well. We start with the following lemma:
\begin{lemma}
\label{lem:semigroup}
$(P_t)_{t\ge0}$ is a strongly continuous contraction semigroup on $\Lpi.$
\end{lemma}
\begin{proof}
By virtue of having shown that $\pi$ is invariant for $P_t$ (part (ii) of Proposition~\ref{prop:X_Markov}) and by virtue of the positivity of $P_t,$ we have that for $f \in \Lpi$,
\[
    \Lpin[ P_tf ]
    \leq 
    \Lpin[ P_t |f| ]
    =(\tfrac{d\pi}{dx}, P_t|f|) 
    =(\tfrac{d\pi}{dx},|f|) 
    = \Lpin[ f ].
  \]
  It follows that $P_t$ is contraction semigroup on $\Lpi$. Now define the space
    \[
    L_0 =\left\{ 
      f \in \Lpi~:~\lim_{t \to 0} \Lpin[ P_t f - f] = 0
    \right\} 
  \]
  By part (iii) of Proposition~\ref{prop:X_Markov}, it contains $C_c^1(0,\infty)$ and is thus dense in $\Lpi$. Moreover, this space is necessarily closed (see \cite[(1.3)]{DynkinMPBookI}), and hence in fact $L_0 = \Lpi.$ This shows that $(P_t)_{t\ge0}$ is also strongly continuous and finishes the proof of the lemma.
\end{proof}

It follows from Lemma~\ref{lem:semigroup} and classical semigroup theory \cite[Theorems~1.1 and~1.3]{DynkinMPBookI} that there is a subspace $\domRG \subset \Lpi$ and an operator $\RG : \domRG \to \Lpi$, called the infinitesimal generator of the semigroup $(P_t)_{t\ge0}$, so that: 
\begin{enumerate}[(i)]
  \item $\domRG$ is the space of all $f \in \Lpi$ such that $\lim_{t \to 0}t^{-1}(P_t f - f)$ exists in $\Lpi$, moreover, this limit equals $\RG f$.
    \item
      For all $\lambda > 0,$ the map $(\lambda - \RG)^{-1}$ given by $\int_0^\infty e^{-\lambda t}P_t\,dt$ is a bounded linear operator from $\Lpi$ to $\domRG.$
    \item
      We have that $P_t$ maps $\domRG \to \domRG$ and for all $f \in \domRG,$ we have that
      \begin{equation}
      \label{eq:Pt_RG}
	P_t f - f = \int_0^t P_s \RG f\,ds = \int_0^t \RG P_s f\,ds.
      \end{equation}
  \end{enumerate}
By part (iii) of Proposition~\ref{prop:X_Markov}, we see that there is no conflict of notation: for $f\in C_c^1(0,\infty)$, the operator $\RG$ defined above indeed coincides with the operator $\RG$ defined in \eqref{eq:RG}. In particular, $C_c^1(0,\infty) \subset \domRG$. The key to the proof of Proposition~\ref{prop:uniqueness} is now the following proposition:
  \begin{proposition}
    The space $C_c^1(0,\infty)$ is a core for $\RG;$ that is, for any $f \in \domRG,$ there is a sequence $\left\{ f_n \right\}_{n=1}^\infty \subset C_c^1(0,\infty)$ so that $f_n \to f$ and $\RG f_n \to \RG f$ in $\Lpi$, as $n\to\infty$.
    \label{prop:core}
  \end{proposition}
  To show this, we will employ an idea used widely in the growth-fragmentation process literature (see for example \cite{Mischler2016} or \cite[Chapter 9]{Banasiak2006}).  We decompose the operator $\RG$ into two pieces $\RGB$ and $\RGK$ where $\RGB$ is roughly a differential part and $\RGK$ is in some sense a perturbation.  Here we define, for all $f \in C_c^1(0,\infty),$
  \begin{align*}
    \RGB f &= xf'(x) -xR(x)f(x) \\
    \RGK f &= xR(x) \int_0^x \frac{2u}{x^2} f(u)\,du.
  \end{align*}
  With this notation, we have that $\RG = \RGB + \RGK.$  We can view $\RGK$ as a perturbation because it is in fact a bounded linear operator on $\Lpi:$
  \begin{lemma}
    $\RGK$ extends uniquely to an operator on $\Lpi.$  Moreover,
    for all $f \in \Lpi,$ we have that 
    \[
      \Lpin[\RGK f] \leq 2\Lpin[f].
    \]
    \label{lem:rgk}
  \end{lemma}
  \begin{proof}
    It suffices to show the estimate in the statement of the Lemma for all $f\in C_c^1(0,\infty).$  For these $f$, we have that by the Fubini-Tonelli theorem,
    \begin{align*}
      \Lpin[\RGK f]
      &\leq \int_0^\infty xR(x) \int_0^x \frac{2u}{x^2} |f(u)|\,du\pi(dx) \\
      &\leq \int_0^\infty 2u |f(u)| \int_u^\infty \frac{R(x)}{x}\pi(dx) \\
      &= \int_0^\infty 2|f(u)|\pi(du),\qquad\text{by \eqref{eq:Fprime}}.
    \end{align*}
This proves the lemma.
  \end{proof}
  
  In light of this, there is a sense in which $\RGB$ and $\RG$ can be compared analytically.  Further, for $\RGB,$ the core statement we wish to prove is relatively straightforward.
  \begin{proposition}
    Let $(Q_t)_{t\ge0}$ be the semigroup 
    \[
      Q_tf(x) = \Exp_x f(X_t) \one[ \tau_1 > t] = f(xe^t)S_x(t),
    \]
    for all bounded Borel measurable $f.$  Then
    \begin{enumerate}[(i)]
      \item $Q_t$ is a strongly continuous contraction semigroup on $\Lpi.$  
      \item For all $f \in C_c^1(0,\infty),$ we have that 
	\[
	  \lim_{t \to 0} t^{-1}(Q_t f - f) = \RGB f,\quad\text{in $\Lpi$.}
	\]
	\end{enumerate}
	Point (ii) says that $C_c^1(0,\infty)$ is a subset of the domain of the infinitesimal generator of the semigroup $(Q_t)_{t\ge0}$ and that $\RGB$ coincides with this generator on $C_c^1(0,\infty)$. We therefore denote without risk of ambiguity by $\RGB$ as well this infinitesimal generator and by $\domRGB\subset \Lpi$ its domain.
	\begin{enumerate}[(i)]\setcounter{enumi}{2}
      \item For any $\lambda > 0,$ the resolvent $(\lambda - \RGB)^{-1} = \int_0^\infty e^{-\lambda t} Q_t\,dt$ maps $\Lpi \to \domRGB$ and has operator norm
	\[
	  \|(\lambda - \RGB)^{-1}\| < \frac{1}{2+\lambda}.
	\]
      \item $C_c^1(0,\infty)$ is a core for $\RGB.$  
    \end{enumerate}
    \label{prop:rgb}
  \end{proposition}
  \begin{proof}
    We begin by observing that for $f \geq 0,$ we have that $Q_t f \leq P_t f,$ at all points.  Hence we have that for all $f\in \Lpi,$
    \[
      \Lpin[ Q_t f ]
      \leq \Lpin[ Q_t |f| ]
      \leq \Lpin[ P_t |f| ] 
      \leq \Lpin[ f ]. 
    \]
    Hence $Q_t$ is indeed a contraction semigroup on $\Lpi.$ Strong continuity follows from the fact that for $f\in C_c(0,\infty)$, we have $Q_tf(x) =  f(xe^t)S_x(t) \to f(x)$ in $\mathrm L^\infty$, hence in $\Lpi$. 
   And hence, we also have it in $\Lpi,$ by the same argument as in Lemma~\ref{lem:semigroup}.
   This proves point (i).
   
   The proof of point (ii) is contained in the proof of part (iii) of Proposition~\ref{prop:X_Markov}.

   We now get to the main part of the proof, which is the proof of the resolvent estimate (iii) and that $C_c^1(0,\infty)$ is a core for $\RGB$ (point (iv)). We will prove them both together. For the statement about the core, we use the following condition from~\cite{EthierKurtz}.
   \begin{lemma}[Chapter 1, Proposition 3.1 of \cite{EthierKurtz}]
     \label{lem:eth}
     Let $A$ generate a strongly continuous contraction semigroup on a Banach space $L.$  Then a subspace $V \subset D_A$ is a core for $A$ if and only if $V$ is dense in $L$ and the range of $\lambda - A$ restricted to $V$ is dense in $L$ for some $\lambda > 0.$
   \end{lemma}
   The density of $C_c^1(0,\infty) \subset \Lpi$ is immediate.  Hence we must show that $\lambda - \RGB$ has dense range for some $\lambda > 0$ to complete the proof.  
   
   Fix $\lambda > 0.$  Suppose we would like to solve the equation
   \begin{align}
     \label{eq:diffeq}
     g = (\lambda - \RGB)f = \lambda f - xf' + xR(x)f(x),
   \end{align}
   for some function $g$.
   This is nothing but a first order linear differential equation, and hence we have the formal solution
   \[
     f(x) = \frac{1}{\mu(x)}
       \int_x^\infty \frac{g(y)}{y}\mu(y)\,dy
   \]
   where $\mu(y) = y^{-\lambda}\exp(-\int_1^y R(x)\,dx).$  Note that we can in fact express $\mu$ in terms of $\pi$ since by \eqref{eq:Fprime_exp} we have that
   \[
     R(z) = - \frac{d}{dz} \log\left( z^{-1} d\pi/dz \right).
   \]
   Hence we may formally express the solution to \eqref{eq:diffeq} by
   \begin{align}
     \label{eq:resol}
     f(x) = \frac{x^{1+\lambda}}{\tfrac{d\pi}{dx}}\int_x^\infty \frac{g(y)}{y^{2+\lambda}}\,\pi(dy)
   \end{align}
(note that $d\pi/dx$ is positive and continuous on $(0,\infty)$, see the discussion around \eqref{eq:Fprime}).
   
   When $g \in C_c(0,\infty),$ then this is a well-defined absolutely continuous solution to \eqref{eq:diffeq}. Hence, $(\lambda-\RGB)^{-1} g = f\in\domRGB$ and %
   \begin{align}
     \Lpin[(\lambda-\RGB)^{-1}g]
     & = \Lpin[f] \nonumber\\
     &\leq \int_0^\infty x^{1+\lambda}\int_x^{\infty} \frac{|g(y)|}{y^{2+\lambda}}\,\pi(dy)\,dx \nonumber\\
     &=\int_0^\infty \frac{|g(y)|}{y^{2+\lambda}}\int_0^{y} x^{1+\lambda}\,dx\,\pi(dy) \nonumber\\
     &=\frac{1}{2+\lambda}\int_0^\infty{|g(y)|}\,\pi(dy) \nonumber\\
  	  &=\frac{1}{2+\lambda}\Lpin[g]\label{eq:resolventbound}
   \end{align}
   By density, we conclude the desired estimate on $\|(\lambda - \RGB)^{-1}\|$ and thus prove point (iii).

   Fixing $g \in C_c(0,\infty),$ we will now show that $f$ can be approximated by a sequence $f_n \in C_c^1(0,\infty)$ so that $(\lambda-\RGB)f_n \to g$ in $\Lpin.$  We start by truncating the support of $f.$ By \eqref{eq:resol}, $f(x) = 0$ for large enough $x$, so that it is enough to truncate near the origin. For $n \in \mathbb{N},$ let $\rho_n$ be a $C_c^1(0,\infty)$ bump function that is $1$ for $x > 3n^{-1}$ and $0$ for $x < 2n^{-1}.$ We can choose these functions so their derivatives are bounded uniformly by $O(n)$ on $[2n^{-1},3n^{-1}].$  In particular, we have that $|\rho_n'(x)| \leq Cx^{-1}$ for some constant $C>0.$ We have,
   \begin{align}
     \label{eq:rgbtrunc}
      \Lpin[ (\lambda - \RGB)(f-f\rho_n)] &
      	\leq 
	\int\limits_0^\infty |\lambda f(x) - xf'(x)|(1-\rho_n(x))\,\pi(dx)\\
	&+
	\int\limits_0^\infty \left|xf(x)\rho_n'(x) \right|\,\pi(dx) \nonumber \\
	&+\int\limits_0^\infty xR(x)|f(x)|(1-\rho_n(x))\,\pi(dx). \nonumber
    \end{align}
  We wish apply dominated convergence to all of these integrals to get their convergence to 0.  We just need to exhibit suitable dominators.  Observe that $\lambda f(x) - xf'(x) = g(x) - xR(x)f(x),$ and $\Lpi[g] < \infty$. Hence, for the first and last integrals, it suffices to show that $\int_0^\infty xR(x)|f(x)|\pi(dx) < \infty.$ Moreover, since $f(x)$ vanishes for large $x$, it is enough to show that this integral converges near zero.

  From \eqref{eq:resol}, we have that $|f(x)|\tfrac{d\pi}{dx} < x^{1+\lambda}C_g$ for all $x \geq 0$ and some constant $C_g.$  Thus it suffices to show that $x^2R(x)$ is integrable at $0.$  Recall from Assumption~\ref{ass:stat} that $\int_0^\infty xR(x)\tfrac{d\pi}{dx}\,dx < \infty.$  
Furthermore, by \eqref{eq:Fprime_exp}, we have for all $x \le 1,$ 
  \[
    \frac{d\pi}{dx} \geq cx,
  \]
  where $c \coloneqq (d\pi/dx)(1)>0$, by the discussion around \eqref{eq:Fprime_exp}. Hence,
  \[
    \int_0^1 cx^2 R(x) \leq \int_0^\infty xR(x)\frac{d\pi}{dx}\,dx < \infty,
  \]
which was to be proven.

  For the second integral of \eqref{eq:rgbtrunc}, since $|\rho_n'(x)| \leq Cx^{-1},$ we can dominate the integrand by $C|f|$, since $\Lpin[f]<\infty$.

  We now argue that it is possible to smooth $f \rho_n$ slightly so that it is in $C_c^1(0,\infty).$  This is relatively straightforward by density.  Just observe that $(f\rho_n)'$ is a compactly supported $\Lpi$ function.  Hence we can choose $f_n \in C_c^1(0,\infty)$ so that 
  \[
    \max\left(\int_0^\infty x|(f\rho_n)'(x) - f_n'(x)|\,\pi(dx) , \|f_n - f\rho_n\|_\infty \right) < \frac{1}{n}.
  \]
  It now follows immediately that
  \(
    (\lambda - \RGB)(f_n - f\rho_n) \to 0
  \)
  in $\Lpin.$  As we also have that 
 \(
    (\lambda - \RGB)(f\rho_n - f) \to 0
  \)
  in $\Lpin,$ the proof is complete.
  \end{proof}

  \begin{proof}[Proof of Proposition~\ref{prop:core}]

    From Lemma~\ref{lem:eth} and Proposition~\ref{prop:rgb} part (iv), there is a $\lambda >0$ so that $\lambda - \RGB$ has dense range when restricted to $C_c^1(0,\infty).$ Then by Proposition~\ref{prop:rgb} part (iii), the resolvent operator $(\lambda - \RGB)^{-1} : \Lpi \to \domRGB$  has operator norm bounded by $(2+\lambda)^{-1}.$ Hence, by Lemma~\ref{lem:rgk}, we have that
\[
  I - \RGK(\lambda - \RGB)^{-1}
\]
is invertible on $\Lpi$ with bounded inverse.   

We now show that the range of $\lambda - \RG$ on $C_c^1(0,\infty)$ is dense in $\Lpi.$  Fix $f \in \Lpi,$ and let $h \in \Lpi$ have that $f = (I - \RGK(\lambda - \RGB)^{-1})h.$  By Proposition~\ref{prop:rgb}, the operator $\lambda - \RGB$ is dense on $C_c^1(0,\infty)$ and so for every $\epsilon > 0$ we can find a $g \in C_c^1(0,\infty)$ so that $\Lpin[ (\lambda - \RGB)g - h] < \epsilon.$  Then we have that
\[
\Lpin[ (\lambda - \RGB - \RGK)g - f] =
\Lpin[ (I - \RGK(\lambda - \RGB)^{-1})((\lambda - \RGB)g - h)] < 2\epsilon,
\]
which completes the proof.
  \end{proof}

We can now finally pass to the proof of the uniqueness result for \eqref{eq:diff_mu}:

\begin{proof}[Proof of Proposition~\ref{prop:uniqueness}]
  By Proposition~\ref{prop:core}, for any $f \in \domRG,$ there is a sequence $f_n \in C_c^1(0,\infty)$ so that 
  \(
f_n \to f
  \)
  and $\RG f_n \to \RG f$ in $\Lpi.$  Since
  $\|\tfrac{d\mu_t}{d\pi}\|_{\operatorname{L}^\infty(\pi)}
  $
  is uniformly bounded, we have that $(\mu_s, \RG f_n) \to (\mu_s, \RG f)$ uniformly in $s \in [0,\infty).$  By taking limits, we obtain:
  \begin{equation}
  \label{eq:diff_mu_full}
   \forall f \in \domRG: (\mu_t, f) - (\mu_0,f) = \int_0^t (\mu_s, \RG f)\,ds.
  \end{equation}

It remains to conclude from \eqref{eq:diff_mu_full} that $\mu_t = \mu_0 P_t$ for all $t\ge0$.  
This follows from a standard argument (see e.g.~\cite[Theorem 1.3]{DynkinMPBookI}), which we recall for convenience.
Set $\nu_t = \mu_0 P_t$ for all $t\ge0$. Then we have for all $f\in \domRG$ and $t\ge0$,
\begin{align*}
 (\nu_t, f) - (\nu_0,f) 
 &= (\nu_0, P_t f - f)\\
 &= (\nu_0, \int_0^t P_s \RG f\,ds) && \text{ by \eqref{eq:Pt_RG}}\\
 &= \int_0^t (\nu_0,P_s\RG f)\,ds\\
 &= \int_0^t (\nu_s,\RG f)\,ds.
\end{align*}
Hence, $(\nu_t)_{t\ge0}$ satisfies \eqref{eq:diff_mu_full}. By linearity, the difference $(\mu_t - \nu_t)_{t\ge0}$ then satisfies \eqref{eq:diff_mu_full} as well.

It now suffices to show that for 
    all $f \in C_c^1(0,\infty)$ and all $\lambda >0,$
    \[
      \int_0^\infty 
      e^{-\lambda t}
      (\mu_t-\nu_t,f)
      \,dt
      = 0,
    \]
    for then it follows that $(\mu_t-\nu_t,f)=0$ for all $t \geq 0$ and thus $\mu_t = \nu_t$ for all $t\ge0$.  Let $g = (\lambda -\RG)^{-1} f \in \domRG \subset \Lpi.$  Then 
    \begin{equation} \label{eq:una}
      \int_0^\infty 
      e^{-\lambda t}
      (\mu_t-\nu_t,f)
      \,dt
      = 
\int_0^\infty 
      e^{-\lambda t}
      (\mu_t-\nu_t,(\lambda - \RG)g)
      \,dt.
    \end{equation}
      Since $(\mu_t-\nu_t)_{t\ge0}$ solves \eqref{eq:diff_mu_full}, a primitive of $(\mu_t-\nu_t,\RG g)$ in $t$ is $(\mu_t-\nu_t,g).$  Note that, by assumption for all $t \geq 0,$
      \[
	|(\mu_t, g)| = |(\pi,g \tfrac{d\mu_t}{d\pi})|
	\leq 
	\|g\|_{\operatorname{L}^1(\pi)}
	\cdot
	\sup_{t \geq 0}\|\tfrac{d\mu_t}{d\pi}\|_{\operatorname{L}^\infty(\pi)}
	< \infty.
      \]
      Note also that for all $t \geq 0,$ by Lemma \ref{lem:semigroup}
      \[
	|(\nu_t, g)| 
	= |(\mu_0,P_t g)|
	\leq \int_0^\infty P_t(|g|) \tfrac{d\mu_0}{d\pi} d\pi 
	\leq \|g\|_{\operatorname{L}^1(\pi)}
	\cdot
	\sup_{t \geq 0}\|\tfrac{d\mu_t}{d\pi}\|_{\operatorname{L}^\infty(\pi)}
	< \infty.
      \]
      Combining the previous two display equations, we conclude that $(\mu_t-\nu_t,g)$ is bounded uniformly in $t.$  Thus for any $\lambda > 0,$ we have using integration by parts that
      \[
	\int_0^\infty 
	e^{-\lambda t}
	(\mu_t-\nu_t,\RG g)
	\,dt
	=
	\int_0^\infty 
	\lambda e^{-\lambda t}
	(\mu_t-\nu_t, g)
	\,dt.
      \]
      We conclude by splitting the right hand side of \eqref{eq:una} and applying the previous display that
      \begin{align*}
      \int_0^\infty 
      e^{-\lambda t}
      (\mu_t-\nu_t,f)
      \,dt
      &= 
\int_0^\infty 
      e^{-\lambda t}
      (\mu_t-\nu_t,\lambda g)
      \,dt 
      +
      \int_0^\infty 
      -\lambda
      e^{-\lambda t}
      (\mu_t-\nu_t,g)
      \,dt \\
      &= 0,
    \end{align*}
    which was to be proven.
\end{proof}

\subsection{Ergodicity}
\label{sec:ergodicity}

Recall the definition of the Markov process $X = (X_t)_{t\ge0}$ from Section~\ref{sec:construction_markov_process} and its associated semigroup $(P_t)_{t\ge0}$. 
The main goal of this section is to prove the following proposition:

\begin{proposition}
\label{prop:ergodicity_gf}
Under Assumption~\ref{ass:stat}, the Markov process $X$ is ergodic, i.e.~for every initial distribution $\nu$, we have
\begin{equation}
\label{eq:ergodicity}
\dtv\left(\nu P_t, \pi\right) \to 0,\quad \text{as $t\to\infty$.}
\end{equation}
Moreover, if $\mathcal M$ is a tight family of distributions, the convergence is uniform on $\mathcal M$, i.e.
\[
\sup_{\nu\in\mathcal M} \dtv\left(\nu P_t, \pi\right) \to 0,\quad \text{as $t\to\infty$.}
\]
\end{proposition}

We could prove this using tools from the theory of piecewise deterministic processes from Costa and Dufour \cite{Costa2008}, which ultimately relies on Meyn and Tweedie \cite{Meyn1993}. However, we prefer a more elementary approach using regeneration which is possible without too many technicalities due to the fact that points are hit almost surely (and are thus \emph{non-polar}).

 Define for $x\in(0,\infty)$,
\[
H_{x} = \inf\{t>0: X_t = x\}.
\]
We will need the following three preliminary lemmas:

\begin{lemma}
 \label{lem:couple_small_time}
For every $t\ge0$ and $x\in(0,\infty)$, denote by $\nu^x_t$ the law of $(X_{s+t})_{s\ge0}$ under $\Pr_{xe^{-t}}$. Then for every $x\in(0,\infty)$,
 \[
 \dtv(\nu^x_t,\nu^x_0) \to 0,\quad\text{as $t\to 0$}.
 \]
\end{lemma}
\begin{proof}
Since $X$ is a Markov process, it suffices to show that the law of $X_t$ under $\Pr_{xe^{-t}}$ converges to $\delta_x$ in total variation. But since $X_t = x$ on the event $\tau_1 > t$ under $\Pr_{xe^{-t}}$, we have,
\[
\Pr_{xe^{-t}}(X_t = x) \ge \Pr_{xe^{-t}}(\tau_1 > t) = S_{xe^{-t}}(t) = \exp\left(-\int_{xe^{-t}}^x R(z)\,dz\right),
\]
and this goes to 1 as $t\to0$.
This proves the claim.
\end{proof}

\begin{lemma}
\label{lem:hit_points}
For any $x\in(0,\infty)$ and $I\subset (0,x]$ compact, there exist $T>0$, such that
\[
\inf_{y\in I} \Pr_y(H_x < T) > 0.
\]
\end{lemma}
\begin{proof}
Fix $x\in(0,\infty)$. Similarly to the proof of Lemma~\ref{lem:couple_small_time}, we have for every $y\in (0,x]$,
\[
\Pr_y(H_x = \log(x/y)) \ge \Pr_y(\tau_1 > \log(x/y)) = \exp\left(-\int_y^x R(z)\,dz\right).
\]
Hence, if $I\subset (0,x]$ is compact and $y_0 = \min I > 0$, then for every $T > \log(x/y_0)$,
\[
\inf_{y\in I} \Pr_y(H_x < T) \ge \exp\left(-\int_{y_0}^x R(z)\,dz\right) > 0.
\]
This shows the result.
\end{proof}

\begin{lemma}
\label{lem:H_abs_cont}
Suppose Assumption~\ref{ass:stat} is in place. Then for any $x\in(0,\infty)$, the law of $H_x$ under $\Pr_x$ has a non-zero absolutely continuous component.
\end{lemma}
\begin{proof}
Fix $x\in(0,\infty)$. Let $f:[0,\infty)\to[0,1]$ be Borel. Then
\begin{align*}
\E_x[f(H_x)] 
&\ge \E_x[f(H_x)\one[\tau_1 \le H_x < \tau_2]].
\end{align*}
By the definition of the process $X$, we have equality of the events
\[
\{\tau_1 \le H_x < \tau_2\} = \{\tau_1 \le -\log(J_1) < \tau_2\},
\]
and on both events we have $H_x = -\log(J_1)$. Applying the strong Markov property at time $\tau_1$, we then get
\[
 \E_x[f(H_x)\one[\tau_1 \le H_x < \tau_2]]= \E_x[\E_{xe^{\tau_1}J_1}[\tau_1 > -\log(J_1)] f(-\log(J_1))\one[-\log(J_1) \ge \tau_1]].
\]
Now note that $\E_{xe^{\tau_1}J_1}[\tau_1 > -\log(J_1)] = S_{xe^{\tau_1}J_1}(-\log(J_1)) > 0$, $\Pr_x$-almost surely. Collecting all the previous equations then gives, for some r.v.~$W>0$,
\[
\E_x[f(H_x)] \ge \E_x[f(-\log(J_1))W \one[-\log(J_1) \ge \tau_1]].
\]
Now note by definition of $J_1$, the law of the r.v. $-\log(J_1)$ is absolutely continuous w.r.t.~Lebesgue measure and has full support in $(0,\infty)$. Also, $J_1$ and $\tau_1$ are independent random variables. Furthermore, by \eqref{eq:R_positivity}, we have $\Pr_x(\tau_1 < \infty) = 1-S_x(\infty) > 0$ (this is the only place where Assumption~\ref{ass:stat} is used). The statement easily follows.
\end{proof}

We can now proceed to the proof of the main result from this section.

\begin{proof}[Proof of Proposition~\ref{prop:ergodicity_gf}]
We first prove the pointwise convergence result \eqref{eq:ergodicity}. This will be shown using a classical result on coupling of regenerative processes. Define a sequence of random times by
\[
\tau_0 = 0,\quad \forall n\ge0:\tau_{n+1} = \inf\{t>\tau_n:X_t = 1\}.
\]
Suppose for the moment we have established the following: for every initial distribution $\nu$, 
\begin{equation}
\label{eq:regeneration}
\Pr_\nu(\forall n:\tau_n < \infty) = 1.
\end{equation}
Then by the strong Markov property, the sequence $(\tau_n)_{n\ge1}$ is a sequence of regeneration times for the process $X$, in particular, $(\tau_n)_{n\ge1}$ is a renewal process with a delay distribution depending on $\nu$. Furthermore, by Lemma~\ref{lem:H_abs_cont}, the interarrival time $\tau_{n+1}-\tau_n$, $n\ge1$, which has the law of $H_1$ under $\Pr_1$, has a non-zero absolutely continuous component. This implies \cite[Chapter 10, Theorem 3.3]{ThorissonBook} that
\[
 \dtv\left(\mu P_t,\nu P_t\right)\to 0,
\]
for all initial distributions $\mu$ and $\nu$. In particular, this is the case for $\mu = \pi$, which satisfies $\pi P_t = \pi$ (Proposition \ref{prop:X_Markov} (ii)) for all $t \geq 0$ and which proves the pointwise convergence result \eqref{eq:ergodicity}.

It remains to show \eqref{eq:regeneration}. This will rely on Lemma~\ref{lem:hit_points} together with Birkhoff's ergodic theorem. We seperate the proof into several steps.

1) By the ergodic theorem, we have for every Borel $A\subset(0,\infty)$,
\begin{equation}
\label{eq:birkhoff}
\frac 1 t \int_0^t \1_{X_s\in A}\,ds \to \E_\pi[\1_{X_0\in A}\,|\,\mathscr I],\quad \text{$\Pr_\pi$-almost surely},
\end{equation}
where $\mathscr I$ is the \emph{invariant $\sigma$-field}, i.e.~the $\sigma$-field of events invariant under time shifts. Fixing a version of the conditional expectation on the right-hand side, this gives for every Borel $A\subset(0,\infty)$:
\begin{equation}
\label{eq:birkhoff2}
\text{for $\pi$-a.e. $x$: } \frac 1 t \int_0^t \1_{X_s\in A}\,ds \to \E_\pi[\1_{X_0\in A}\,|\,\mathscr I],\quad \text{$\Pr_x$-almost surely},
\end{equation}

2) We now show that $\mathscr I$ is trivial\footnote{This seems to be a general property of positive Harris recurrent Markov processes (maybe under some extra irreducibility assumption), but we could not find a proper reference.} under $\Pr_\pi$ (although one could work around it). Taking expectations in \eqref{eq:birkhoff2} and using dominated convergence gives for every Borel $A\subset(0,\infty)$, 
\[
\text{for $\pi$-a.e. $x$: } \frac 1 t \int_0^t \Pr_x(X_s\in A)\,ds \to \pi(A),
\]
In fact, the above statement is true for every $x$, which is an easy consequence of Lemma~\ref{lem:couple_small_time} and the fact that $\pi$ has full support in $(0,\infty)$ (see the discussion after \eqref{eq:Fprime_exp}). This implies
\[
\text{for all $x\in(0,\infty)$: } \dtv\left(\frac 1 t \int_0^t \delta_xP_s\,ds, \pi \right) \to 0,
\]
which, by \cite[Chapter 6, Theorem 5.2]{ThorissonBook} proves triviality of $\mathscr I$ under $\Pr_\pi$. 

3) By the triviality of $\mathscr I$ under $\Pr_\pi$, Equations \eqref{eq:birkhoff} and \eqref{eq:birkhoff2} hold with the right-hand side replaced by $\pi(A)$. Using again Lemma~\ref{lem:couple_small_time}, this implies the following for every Borel $A\subset(0,\infty)$:
\begin{equation}
\label{eq:birkhoff3}
\text{for all $x\in(0,\infty)$: } \frac 1 t \int_0^t \1_{X_s\in A}\,ds \to \pi(A),\quad \text{$\Pr_x$-almost surely.}
\end{equation}

4) Since $\pi$ has full support, we have $\pi([1/2,1]) > 0$. It follows from \eqref{eq:birkhoff3} that for every starting distribution $\nu$, the interval $[1/2,1]$ is visited a positive proportion of times $\Pr_\nu$-almost surely. Together with Lemma~\ref{lem:hit_points} and the strong Markov property, this easily allows to show that the point $1$ is visited an infinite number of times $\Pr_\nu$-almost surely. This is exactly \eqref{eq:regeneration}, which was to be shown.

\medskip

We have shown \eqref{eq:ergodicity}. We now show that the convergence is uniform on tight families of distributions on $(0,\infty)$. Define for all $t \geq 0$ and all $x > 0,$
  \[
    m_t(x) = \dtv\left( 
    \delta_x P_t, \pi
    \right).
  \]
  As $P_t$ is a contraction in the total variation distance, this function is monotone decreasing in $t$. Further, by \eqref{eq:ergodicity}, it converges to $0$ pointwise as $t \to \infty.$  We claim that the convergence is uniform on compact sets. For this, it is enough to show the same for the function
  \[
  \widetilde m_t(x) = \begin{cases}
  m_{t-\log x}(x), & x\le e^t\\
  1, & \text{otherwise}.
  \end{cases}
  \]
 This function converges again pointwise to 0 as $t\to\infty$ and is monotone decreasing in $t$. Further, using Lemma~\ref{lem:couple_small_time} it is easy to show that $\widetilde m_t(x)$ is continuous in $x$ for every $x < e^t$. By Dini's theorem, $\widetilde m_t$ therefore converges to $0$ as $t\to\infty$, uniformly on compact sets, and the same easily follows for $m_t$. 
 
 Now let $\mathcal M$ be a tight family of distributions on $(0,\infty)$. Fix $\epsilon > 0.$ We wish to show that $\sup_{\nu\in\mathcal M} \dtv\left( \nu P_t, \pi \right)<\epsilon$ for every $t$ large enough. By tightness of $\mathcal{M}$ we can find a compact set $K \subset \left( 0,\infty \right)$ such that for any $\nu \in \mathcal{M},$ $\nu(K^c) < \epsilon/2.$  We then bound
  \begin{align*}
    \dtv\left( \nu P_t, \pi \right) 
    &\leq
    \int\limits_0^\infty m_t(x)\,\nu(dx) \\
    &\leq 
    \int\limits_K m_t(x)\,\nu(dx) + \epsilon/2 \\
    &\leq \sup_{x \in K} m_t(x) + \epsilon/2.
  \end{align*}
The first term on the right-hand side is smaller than $\epsilon/2$ for large $t$ by the uniform convergence on compact sets of $m_t$. Since $\epsilon$ was arbitrary, this shows that $\sup_{\nu\in\mathcal M} \dtv\left( \nu P_t, \pi \right)\to 0$ goes to $0$ as $t\to\infty$, which was to be proven.
\end{proof}

\subsection{Proof of Proposition~\ref{prop:convergence_RFPSpace}}
\label{sec:proof_prop}

We now wrap up the results from the previous sections for the proof of Proposition~\ref{prop:convergence_RFPSpace}:

\begin{proof}[Proof of Proposition~\ref{prop:convergence_RFPSpace}]

Let $\F\in\RFPSpacea$ and define $\mu_t = d\F[t]$. By Proposition~\ref{prop:markov}, the family of measures $(\mu_t)_{t\ge0}$ satisfies \eqref{eq:diff_mu} for every $f\in C_c^1(0,\infty)$. Furthermore, by definition of the space $\RFPSpacea$, the family $(\mu_t)_{t\ge0}$ satisfies the first assumption of Proposition~\ref{prop:uniqueness}. Proposition~\ref{prop:uniqueness} then shows that $\mu_t = \mu_0 P_t$ for all $t\ge0$. The result then follows from Proposition~\ref{prop:ergodicity_gf}.
\end{proof}

\section{Continuity of $\COP$}
\label{sec:proof_continuity}
Recall the notion of locally uniform convergence on the space  $\DTSpace = \mathcal B([0,\infty),\DSpace)$ of Borel measurable maps from $[0,\infty)$ to $\DSpace$: $\F^{(n)} \Dto \F$ as $n\to\infty$ if and only if for all compact $K \subseteq [0,\infty)$ and all $T> 0,$
\[
\lim_{n \to \infty} \sup_{0 \leq t \leq T} \int\limits_K | \F[t]^{(n)}(x) - \F[t](x) |\,dx = 0.
\]
For clarity, we define a finer topology on $\DTSpace$ by saying $\F^n \DSto \F$ if for all $T > 0$
  \[
    \lim_{n \to \infty}
    \sup_{0 \leq s \leq T}
    \sup_{0 \leq x}
    | \F[s]^n(x) - \F[s](x)| \to 0,
  \]
which is metrizable in a similar way as was done for $\dblah.$  We let $\DTSSpace$ denote the resulting metric space.
\begin{lemma}
  Suppose that $\F^n \Dto \F,$ where $(t,x) \mapsto \F[t](x)$ is a continuous map from $[0,\infty)^2 \to [0,1]$ and for each $t,$ $\F[t](+\infty)=1$, then $\F^n \DSto \F.$
  \label{lem:conv_upgrade}
\end{lemma}
\begin{proof}
  Fix $T>0$ and $\epsilon>0$ and define the sequence of functions $f_n(t) = F_t(n)$ on $[0,T]$. By assumption, the functions $f_n$ are continuous and converge pointwise to $1$ as $n\to\infty$. By monotonicity, Dini's theorem implies that the convergence holds uniformly on $[0,T]$. Hence, we can find a $K > 0$ sufficiently large that $\F[t](K) > 1 - \epsilon/2$ for all $0 \leq t \leq T.$  By the continuity of $(t,x) \mapsto \F[t](x),$ we have that there exists a $\delta > 0$ sufficiently small that $|x-y| < \delta$ imply that $|\F[t](x) - \F[t](y)| < \epsilon/2,$ for all $t\in[0,T]$ and $x,y\in[0,K]$. Using that $\F[t](x) > 1-\epsilon/2$ for all $x\ge K$, it is easily shown that
\begin{align}
\label{eq:equicont}
|\F[t](x) - \F[t](y)| < \epsilon\quad\text{$\forall t\in[0,T],\,x,y\ge 0$ s.t. $|x-y| < \delta$}.
\end{align}
  
Now let $t\in[0,T]$ and $x\ge0$ such that $\F[t]^n(x) \ge \F[t](x) + 2\epsilon$. Note that necessarily, $x\le K$, otherwise $\F[t]^n(x) > 1+\epsilon$. Then by monotonicity of $\F[t]^n(x)$ in $x$ and \eqref{eq:equicont}, we have $\F[t]^n(y) > \F[t](y)+\epsilon$ for all $x < y < x+\delta.$  Hence, 
  \begin{align*}
    \forall t\in[0,T]:\int_0^{K+\delta} |\F[t]^n(x) - \F[t](x)|\,dx \ge \epsilon\delta.
  \end{align*}
  Thus, since $\sup_{0\le t\le T} \int_0^{K+\delta} |\F[t]^n(x) - \F[t](x)|\,dx\to 0$ as $n\to\infty$, we must have
\[
\limsup_{n\to\infty} \sup_{0\le t\le T} \sup_{x\ge0} \F[t]^n(x) - \F[t](x) \le 2\epsilon.
\]  

A similar argument shows that
\[
\liminf_{n\to\infty} \inf_{0\le t\le T} \inf_{x\in [\delta,K]} \F[t]^n(x) - \F[t](x) \ge -2\epsilon.
\]
By the monotonicity of $\F[t]^n$ and the fact that $\F[t](x) \ge 1-\epsilon$ for all $x\ge K$, we also have
\[
\liminf_{n\to\infty} \inf_{0\le t\le T} \inf_{x\ge K} \F[t]^n(x) - \F[t](x) \ge -3\epsilon.
\]
Furthermore, since $\F[t]^n\ge 0$ and $\F[t](x) \in [0,\epsilon]$ for all $x\in[0,\delta]$ by \eqref{eq:equicont}, we have
\[
\liminf_{n\to\infty} \inf_{0\le t\le T} \inf_{x\in [0,\delta]} \F[t]^n(x) - \F[t](x) \ge -\epsilon.
\]
Altogether, this shows that
\[
\limsup_{n\to\infty} \sup_{0\le t\le T} \sup_{x\ge0} |\F[t]^n(x) - \F[t](x)| \le 3\epsilon.
\]
As $\epsilon$ was arbitrary, we have that $\F^n \DSto \F.$
\end{proof}

This improved convergence allows us to control the Radon--Nikodym derivatives.

\begin{lemma}
  Suppose that $\Psi \in C^1[0,1]$ with derivative $\psi.$
  Suppose that $\F^n \DSto \F,$ where $(t,x) \mapsto \F[t](x)$ is a continuous map from $[0,\infty)^2 \to [0,1]$ with $\left\{ \F[t] : t \geq 0 \right\}$ tight.
  Then
  \[
    \lim_{n \to \infty} \sup_{0 \leq s \leq T} \int_0^\infty
    \left|
    \frac{d\Psi(\F[s]^n)}{d\F[s]^n}(x)
    -\psi(\F[s](x))
    \right|
    \,d\F[s]^n(x)
    =0
  \]
  \label{lem:rnderiv}
\end{lemma}
\begin{proof}
It is enough to prove the statement with $\psi(\F[s](x))$ replaced by $\psi(\F[s]^n(x))$, as is easily shown using the uniformity in the definition of the convergence $\F^n \DSto \F$. 
  Define for any $a \in [0,1]$ and any $\delta > 0$
  \[
    \psi_\delta(a) = \sup_{\substack{|a-x| < \delta \\ |a - y| < \delta}} 
    \frac{\Psi(x)-\Psi(y)}{x-y},
  \]
  where $x,y \in [0,1].$  %
  For any distribution function $F$ all of whose atoms are at most $\delta/2,$ there is a $\eta_0$ sufficiently small such that for all $a \geq 0$ and all $0 < \eta < \eta_0$
  \[
    F(a+\eta)-F(a) < \delta.
  \]
  Then by definition of $\psi_\delta,$
  \begin{align*}
    \int_a^{a+\eta} \psi_\delta(F(x))\,dF(x)
    &\geq 
    \int_a^{a+\eta}\frac{\Psi(F(a+\eta)) - \Psi(F(a))}{F(a+\eta) - F(a)}\,dF(x)\\
    &=\Psi(F(a+\eta)) - \Psi(F(a))
  \end{align*}
  It follows by taking unions that for any $U = (a,b]$ with $a<b$
  \[
    \int_U \psi_\delta(F(x))\,dF(x) > \int_U \frac{d\Psi(F)}{dF}(x)\,dF(x),
  \]
  and hence by a monotone class argument we have $\psi_\delta(F(x)) \geq \frac{d\Psi(F)}{dF}(x)$ $dF$--almost everywhere. %
  
  From the convergence of $\F^n \DSto \F,$ the continuity of $\F$ and the tightness of $\F[t],$ for any $\delta >0$ there is an $n_0$ sufficiently large and a $\eta_0$ sufficiently small that for $n \geq n_0$ and $0 < \eta <\eta_0,$
  \[
    \sup_{a \geq 0} \sup_{0 \leq s \leq T} (\F[s]^n(a+\eta) - \F[s]^n(a)) < \delta.
  \]
  Hence by the same argument as above, for any $0 \leq s \leq T$
  \[
    \frac{d\Psi(\F[s]^n)}{d\F[s]^n}(x) \leq \psi_\delta(\F[s]^n(x)).
  \]
  For any $\epsilon > 0$, we can pick $\delta$ sufficiently small that 
  $\psi_\delta(p) \leq \psi(p) + \epsilon$ for all $p \in [0,1].$
Then,
  \[
    \limsup_{n \to \infty} \sup_{0 \leq s \leq T} \int_0^\infty
    \left(
    \frac{d\Psi(\F[s]^n)}{d\F[s]^n}(x)
    -\psi(\F[s]^n(x))
    \right)_+   \,d\F[s]^n(x)
    \leq \epsilon.
  \]
  Taking $\epsilon \to 0$ shows this $\limsup$ is at most 0.

  An almost identical proof using
  \[
    \psi^{-}_\delta(a) = \inf_{\substack{|a-x| < \delta \\ |a - y| < \delta}} 
    \frac{\Psi(x)-\Psi(y)}{x-y},
  \]
  shows
  \[
    \limsup_{n \to \infty} \sup_{0 \leq s \leq T} \int_0^\infty
    \left(
    \frac{d\Psi(\F[s]^n)}{d\F[s]^n}(x)
    -\psi(\F[s]^n(x))
    \right)_{-}  \,d\F[s]^n(x)
    \leq 0.
  \]
  By what was mentioned at the beginning of the proof, this implies the statement of the lemma.
\end{proof}

Using this additional strengthening of the mode of convergence, we now show the continuity of $\COP$ at those limit points at which the first time evolving distribution function is continuous.
\begin{proof}[Proof of Proposition~\ref{prop:continuity}]
Let $(\F^n,\G^n)_{n\ge1}$ and $(\F,\G)$ be as in the statement of the proposition.
  We must show that for any fixed $T, K >0,$
  \[
    \lim_{n \to \infty}
    \sup_{0 \leq t \leq T} 
    \int\limits_0^K
    |\COP(\F^n,\G^n)_t(x)
    -\COP(\F,\G)_t(x)|\,dx = 0.
  \]
We bound $\COP(\F^n,\G^n)_t(x)-\COP(\F,\G)_t(x)$ pointwise by the following pieces:
\begin{align}
|\COP(\F^n&,\G^n)_t(x)-\COP(\F,\G)_t(x)| \nonumber \\ 
&\leq |\G[0]^n(e^{-t}x) - \G[0](e^{-t}x)| \label{eq:cop1} \\
&+\int_0^t (e^{s-t}x)^2\int_{e^{s-t}x}^\infty
{\left|\frac{d\Psi(\F[s]^n)}{d\F[s]^n}(z) - \psi(\F[s](z))\right|}
\frac{d\G[s]^n(s)}{z}
\,ds \label{eq:cop2} \\
&+\int_0^t (e^{s-t}x)^2
\left|\int_{e^{s-t}x}^\infty
\frac{\psi(\F[s](z))}{z}
(
d\G[s]^n(z) 
-d\G[s](z) 
)
\right|\,ds.
\label{eq:cop3}
\end{align}
The first piece \eqref{eq:cop1} converges to $0$ simply by definition of $\G^n \Dto \G.$  As for the second piece \eqref{eq:cop2}, we can bound
\begin{align*}
\int_0^t (e^{s-t}x)^2&\int_{e^{s-t}x}^\infty
{\left|\frac{d\Psi(\F[s]^n)}{d\F[s]^n}(z) - \psi(\F[s](z))\right|}
\frac{d\G[s]^n(s)}{z}\\
&\leq
\sup_{\substack{0 \leq s \leq t \\ 0 \leq z}} 
\left|\frac{d\Psi(\F[s]^n)}{d\F[s]^n}(z) - \psi(\F[s](z))\right|
\int_0^t (e^{s-t}x)\int_{e^{s-t}x}^\infty
\,d\G[s]^n(z)ds \\
&\leq
x
\sup_{\substack{0 \leq s \leq t \\ 0 \leq z}} 
\left|\frac{d\Psi(\F[s]^n)}{d\F[s]^n}(z) - \psi(\F[s](z))\right|.
\end{align*}
Hence by Lemma~\ref{lem:rnderiv}, this converges to $0$ in $\DTSpace.$

Finally, for the last piece \eqref{eq:cop3}, we let, for any $s,u \geq 0,$
\(
Q_s^n(u) = \int_u^\infty \frac{\psi(\F[s](z))}{z}\,d\G[s]^n(z),
\)
and let $Q$ denote the same with $\G^n$ replaced by $\G.$  By weak convergence of $\G[s]^n$ to $\G[s],$ at every point of continuity of $\G[s],$ we have that $Q_s^n(u) \to Q_s(u)$ as $n \to \infty.$

Hence, we have that at fixed $t,$
\begin{align*}
  \int_0^K
  \int_0^t (e^{s-t}x)^2
&\left|\int_{e^{s-t}x}^\infty
\frac{\psi(\F[s](z))}{z}
(
d\G[s]^n(z) 
-d\G[s](z) 
)
\right|\,ds\,dx \\
&=
  \int_0^K
  \int_0^t (e^{s-t}x)^2
  \left|Q^n_s(e^{s-t}x) - Q_s(e^{s-t}x)\right|\,ds\,dx \\
&=
  \int_0^t
  \int_0^K
  (e^{s-t}x)^2
  \left|Q^n_s(e^{s-t}x) - Q_s(e^{s-t}x)\right|\,dx\,ds \\ 
&=
  \int_0^t
  e^{t-s}
  \int_0^{e^{t-s}K}
  u^2
  \left|Q^n_s(u) - Q_s(u)\right|\,du\,ds.
\end{align*}
Let $H(s,t)^n$ be given by
\[
  H(s,t)^n = \int_0^{e^tK} x^2 \left|Q^n_s(x) - Q_s(x)\right|\,dx.
\]
Observe that we can bound $Q_s^n(u) \leq \frac{\|\psi\|_\infty}{u}.$  Hence by dominated convergence, for every fixed $s,t$ with $0\leq s \leq t,$
\[
  \lim_{n \to \infty} H(s,t)^n = 0.
\]
Further, we have that $H(s,t)^n$ can be bounded uniformly by $2e^{2t}K^2\|\psi\|_\infty.$  Hence again by dominated convergence, we have that
\begin{align*}
  \sup_{0 \leq t \leq T}
  \int_0^K
  \int_0^t (e^{s-t}x)^2
&\left|\int_{e^{s-t}x}^\infty
\frac{\psi(\F[s](z))}{z}
(
d\G[s]^n(z) 
-d\G[s](z) 
)
\right|\,ds\,dx \\
&\leq  \int_0^T e^{T} H(s,T)^n\,ds,
\end{align*}
which converges to $0$ by dominated convergence.
\end{proof}

\printbibliography

\end{document}